\documentclass[twoside, 10pt, letter]{article}

\usepackage{amsmath, amscd, amsfonts, amssymb, amsthm, latexsym, url}

\usepackage{graphicx}

\usepackage[left=1.3in, right=1.3in, top=1.3in, bottom=1.3in, includefoot, headheight=13.6pt]{geometry}

\input{xy}
\xyoption{all}
\objectmargin={2mm}

\usepackage{fancyhdr}

\usepackage{sectsty}
\sectionfont{\Large\sc\centering}
\subsectionfont{\sc}
\partfont{\huge\centering\sc}

\newtheorem{theorem}{Theorem}[section]
\newtheorem{lemma}[theorem]{Lemma}
\newtheorem{corollary}[theorem]{Corollary}
\newtheorem{proposition}[theorem]{Proposition}

\theoremstyle{definition}
\newtheorem{definition}[theorem]{Definition}
\newtheorem{remark}[theorem]{Remark}
\newtheorem{example}[theorem]{Example}

\newcommand{\xysquare}[8]{
\[\xymatrix{
#1 \ar@{#5}[r] \ar@{#6}[d] & #2 \ar@{#7}[d]\\
#3 \ar@{#8}[r] & #4
}\]
}

\newcommand{\al}{\alpha}
\newcommand{\bb}{\mathbb}

\newcommand{\bor}{\partial}

\newcommand{\Cl}[1]{\overline{\{#1\}}}
\newcommand{\comment}[1]{}

\newcommand{\into}{\hookrightarrow}
\newcommand{\isoto}{\stackrel{\simeq}{\to}}

\newcommand{\mult}[1]{#1^{\!\times}}

\newcommand{\op}{\operatorname}
\newcommand{\pid}[1]{\langle #1 \rangle}

\newcommand{\roi}{\mathcal{O}}

\newcommand{\sub}[1]{{\mbox{\scriptsize #1}}}

\newcommand{\To}{\longrightarrow}
\newcommand{\ul}[1]{\underline{#1}}

\newcommand{\xto}{\xrightarrow}

\renewcommand{\cal}{\mathcal}
\renewcommand{\hat}{\widehat}
\renewcommand{\frak}{\mathfrak}
\newcommand{\indlim}{\varinjlim}
\renewcommand{\tilde}{\widetilde}
\renewcommand{\Im}{\operatorname{Im}}
\renewcommand{\ker}{\operatorname{Ker}}
\renewcommand{\projlim}{\varprojlim}

\DeclareMathOperator{\Frac}{Frac}

\DeclareMathOperator{\Hom}{Hom}

\DeclareMathOperator*{\rprod}{\prod\nolimits^{\prime}\hspace{-1mm}}

\DeclareMathOperator{\Spec}{Spec}

\DeclareMathOperator*{\projlimf}{``\varprojlim''}
\DeclareMathOperator*{\holim}{\operatorname*{holim}}

\newcommand{\comp}{{\hat{\phantom{o}}}}

\begin{document}
\title{A singular analogue of Gersten's conjecture and applications to $K$-theoretic ad\`eles}

\author{\sc Matthew Morrow}

\date{}



\maketitle
\begin{abstract}
The first part of this paper introduces an analogue, for one-dimensional, singular, complete local rings, of Gersten's injectivity conjecture for discrete valuation rings. Our main theorem is the verification of this conjecture when the ring is reduced and contains $\bb Q$, using methods from cyclic/Hochschild homology and Artin-Rees type results due to A.~Krishna.

The second part of the paper describes the relationship between ad\`ele type resolutions of $K$-theory on a one-dimensional scheme and more classical questions in $K$-theory such as localisation and descent. In particular, we construct a new resolution of sheafified $K$-theory, conditionally upon the conjecture.
\end{abstract}

\section{Introduction}
Suppose that $A$ is a one-dimensional local ring; letting $\frak m$ denote its maximal ideal, consider the `completed $K$-group' \[\hat K_n(A):=\projlim_r K_n(A/\frak m^r).\] These groups appeared first perhaps in work by J.~Wagoner \cite{Wagoner1975, Wagoner1976, Wagoner1976a} for complete discrete valuation rings, where they were defined in a different, but equivalent, fashion. The following conjecture is explored in the first part of this paper:
\begin{quote}
If $A$ is a one-dimensional, complete, Noetherian local ring, then the diagonal map \[K_n(A)\to \hat K_n(A)\oplus K_n(\Frac A)\] is injective, where $\Frac A$ denotes (assuming $A$ is Cohen-Macaualy) the total quotient ring of $A$.
\end{quote}
In other words, the $K$-theory of the disk $\Spec A$ is determined on the punctured disk together with all infinitesimal thickenings of the closed point. Our main theorem is that the conjecture is true if $A$ is reduced and contains $\bb Q$ (see footnote\footnote{At the time of publishing, the author has improved the result by showing that if $A$ is reduced and contains $\bb Q$ then $K_n(A)\into K_n(A/\frak m^r)\oplus K_n(\Frac A)$ for $r\gg 0$.}).

The conjecture is not surprising if $A$ is regular, for then the Gersten conjecture (a theorem in many cases) predicts already that $K_n(A)\to K_n(\Frac A)$ is injective; of course, it would still be interesting to have a proof of the conjecture for those discrete valuation rings for which the Gersten conjecture remains unknown, but that it is not our goal. Rather we are claiming that when $A$ is singular, the failure of $K_n(A)\to K_n(\Frac A)$ to be injective is captured entirely by the $K$-theory of all the quotients $A/\frak m^r$, $r\ge 1$, as long as $A$ is complete (or Henselian, as we shall see).

An informative example is provided by taking $A$ to be the completion of the local ring of a seminormal, rational singularity on a curve over a field, and $n=2$. Classical calculations due to R.~Dennis and M.~Krusemeyer \cite{Dennis1979}, C.~Weibel \cite{Weibel1980}, and S.~Geller \cite{Geller1986} show that the kernel of $K_2(A)\to K_2(\Frac A)$ is non-zero and that it embeds into $K_2(A/\frak m^2)$, verifying the conjecture in this case. See proposition \ref{proposition_seminormal}.

Before turning to the second part of the paper, on global theory, we describe more precisely the layout of the first part. Section \ref{section_defs} contains the main definitions, various remarks on the conjecture, and some theoretical tools for passing between complete rings and Henselian ones. It also summarises the main results and provides a counterexample showing that completeness/Henselianess is essential in the conjecture.

Section \ref{section_main_proofs} is the proof of the main theorem, namely verification of the conjecture when $A$ is reduced and contains $\bb Q$ (or is truncated polynomials over such a ring). Results and ideas from two papers by A.~Krishna \cite{Krishna2005, Krishna2010}, concerning Artin-Rees type properties in Hochschild and cyclic homology, are absolutely essential.  We use the standard comparisons between the $K$-theory and cyclic homology of $\bb Q$-algebras, namely T.~Goodwillie's \cite{Goodwillie1986} result on nilpotent extensions and G.~Corti\~nas' proof \cite{Cortinas2006} of the KABI conjecture, and we compare $K$-groups/cyclic homology groups of $A$ with those of its normalisation $\tilde A$ (which is smooth, so its cyclic homology is well understood). This only works because Krishna's results imply that $K$-theory and cyclic homology satisfy excision when we take the limit over $\frak m^r$, $r\ge 1$.

Section \ref{section_example} contains miscellaneous structural results and examples concerning $\hat K_n(A)$. When $A$ has finite residue field, we show that $\hat K_n(A)$ is a profinite group and we offer two alternative homotopy theoretic descriptions of it (the second is only allowed in the mixed characteristic case): $\hat K_n(A)\cong \pi_n(\holim_rK(A/\frak m^r))\cong \pi_n(K(A)^\comp)$. We also apply Moore's theorem to completely describe $\hat K_2(\roi)$ when $\roi$ is the ring of integers in a finite extension of $\bb Q_p$. On the other hand, we show that if $A$ is a discrete valuation ring of residue characteristic zero, then $\hat K_n(A)$ differs from $K_n$ of its residue field by a `pile of differential forms'.

The second part of the paper focuses on global constructions, reformulating aspects of $K$-theory of one-dimensional schemes from an adelic point of view. As motivation, we briefly now review the usual id\`eles\footnote{There is a vague convention to use the word `ad\`eles' for additive type objects, and `id\`eles' for multiplicative type ones; but this is so badly defined that we will prefer to speak of ad\`eles for anything other than the original group of id\`eles.} from a geometric perspective: given a one-dimensional (Noetherian, and temporarily regular for simplicity) scheme $X$, let \[\rprod_{x\in X_0} \mult F_x\] be the restricted product of the unit groups of the fractions of $\roi_x:=\hat\roi_{X,x}$, for $x\in X_0$; this is the familiar ring of (finite) id\`eles if $X$ is the spectrum of the ring of integers of a number field. It is not hard to check that the cohomology of the complex (which is the reduced complex attached to a simplicial group) \[0\To \prod_x\mult \roi_x\oplus \mult{K(X)}\To \rprod_x \mult F_x\To 0\tag{\dag}\] is precisely $H^*(X,\bb G_m)$. In the second part of the paper we extend this result to higher degree $K$-theory, in such a way that the local factors of the adelic complex are the completed $K$-groups studied in the first part; the main result is theorem \ref{theorem_main}. However, the journey is as important as the final result, as we will explain in the following summary.

Section \ref{section_Zariski} starts by describing the general theory of (incomplete) ad\`eles on a curve $X$ for an arbitrary sheaf $\cal F$ on abelian groups on $X_\sub{Zar}$; this is by no means new, but I do not know of any reference. In particular, we carefully define our `restricted product' notation $\rprod$\,.

Sections \ref{subsection_Zariski_les} and \ref{subsection_Zariski_descent} were inspired by calculations in C.~Weibel's paper \cite{Weibel1986}, in which truncations of ad\`eles for the sheaf $\cal K_n$ ($=$ Zariski sheafification of $K_n$) already appear. We derive a long exact Mayer-Vietoris sequence \[\cdots\to K_n(X)\to\prod_{x\in X^1}K_n(\roi_{X,x})\oplus\prod_{y\in X^0}K_n(\roi_{X,y})\to\rprod_{x\in X^1}K_n(\Frac\roi_{X,x})\to\cdots\tag{\ddag}\] relating the (incomplete) $K$-theoretic ad\`eles with the $K$-theory of $X$ itself. This arises from taking an increasingly fine limit over $K$-theory localisation sequences (in the style of R.~Thomason and T.~Trobaugh \cite{Thomason1990}). We show that the existence of such a long exact sequence is essentially equivalent to $K$-theory satisfying descent on $X_\sub{Zar}$.

These arguments are then repeated, almost verbatim, in section \ref{section_Nisnevich_case} for the Nisnevich topology on $X$. In particular, we obtain another long exact Mayer-Vietoris sequence
\[\cdots\to K_n(X)\to\prod_{x\in X^1}K_n(\roi_{X,x}^h)\oplus\prod_{y\in X^0}K_n(\roi_{X,y})\to\rprod_{x\in X^1}K_n(\Frac\roi_{X,x}^h)\to\cdots,\]
where this time the local factors are $K$-groups of Henselizations; as in the Zariski case, its existence is equivalent to $K$-theory satisfying descent on $X_\sub{Nis}$. In section \ref{subsection_Nisnevich_via_completions} it is explained how, in our local-to-global formulae describing the Nisnevich cohomology of $K$-theory, one can always work with completions of the local rings, $\hat \roi_{X,x}$, rather than their Henselisations (c.f.~corollary \ref{corollary_equivalence_of_completed_and_Henselian}).

In section \ref{subsection_Nisnevich_cohomology_via_adeles} we reach our goal, namely theorem \ref{theorem_main}, which is conditional on the aforementioned conjecture being satisfied for $\hat\roi_{X,x}$ for all $x\in X^1$: Firstly, there is a long exact Mayer-Vietoris sequence, similar to (\ddag), where the local factors are completed $K$-groups. Secondly, the cohomology of a complex like (\dag), but again with completed $K$-groups instead, computes $H^*_\sub{Nis}(X,\cal K_n)$.

We finish this introduction by commenting that the longer term goal of this work is to develop a theory of $K$-theoretic ad\`eles in arbitrary dimensions in the spirit of A.~Parshin and A.~Beilinson's theory of higher ad\`eles \cite{Beilinson1980, Parshin1976}. This would (conjecturally) offer a more functorial alternative to the Gersten resolution of $K$-theory; moreover, it would continue to work on singular schemes. The recursive fashion by which ad\`eles are constructed in higher dimensions forces one to consider singular, non-reduced schemes in dimension one; moreover, it is expected that functoriality in higher dimensions will use K.~Kato's residue homomorphisms on completed $K$-groups \cite{Kato1980, Kato1983}, locally representing the pushforward of cycles along proper morphisms. Therefore a preliminary study of the completed $K$-groups of one-dimensional, singular local rings was necessary.

The appendix contains summaries of some facts from $K$-theory and Hochschild/cyclic homology, and collects together some classical results on the $K$-theory of seminormal rings.

\subsection*{Acknowledgements}
I am extremely grateful to C.~Weibel for numerous helpful conversations.

\part*{Part I: Local theory}
\section{An singular analogue of Gersten's conjecture}\label{section_defs}
Let $A$ be a one-dimensional Noetherian local ring, typically singular (all rings will be Noetherian, so we will not mention this hypothesis again); its maximal ideal will be denoted $\frak m=\frak m_A$. Throughout this paper, we will write $\Frac A$ to mean \[\Frac A:=\prod_\frak pA_\frak p,\] where $\frak p$ runs over the minimal prime ideals of $A$; so $\Spec(\Frac A)$ is the punctured spectrum $\Spec A\setminus\frak m$ (we will make some comments about this notation in remark \ref{remark_conjecture_for_non_CM} below).

Of central interest in this work are the {\em completed $K$-groups} of $A$, namely \[\hat K_n(A):=\projlim_r K_n(A/\frak m^r).\] The reader interested in seeing some examples immediately may wish to glance at section \ref{section_example}. The conjecture we will explore is the following:
\begin{quote}
{\bf Conjecture 1:} If $A$ is a one-dimensional, complete local ring, then the diagonal map \[K_n(A)\to \hat K_n(A)\oplus K_n(\Frac A)\] is injective for all $n\ge 0$.
\end{quote}
We have no counterexamples to this, and will prove various special cases.

\begin{remark}
On occasion $A$ will merely be semi-local, in which case $\Frac A$ is defined as above, we use $\frak M=\frak M_A$ for the Jacobson radical, and we put $\hat K_n(A)=\projlim_rK_n(A/\frak M^r)$. But if $A$ is a complete semi-local ring then it is a finite product of complete local rings, so the conjecture trivially extends to the semi-local case.
\end{remark}

\begin{remark}
The conjecture is trivial if $n=0$ or $1$. If $A$ is regular, i.e.~a discrete valuation ring, then Gersten's conjecture predicts that already $K_n(A)\to K_n(\Frac A)$ is injective. Conversely, when $A$ is singular the map $K_n(A)\to K_n(\Frac A)$ usually has non-zero kernel (c.f.~Geller's conjecture \cite{Geller1986}), and conjecture 1 claims that this kernel is `small enough' to inject into $\hat K_n(A)$ (when $A$ is complete).
\end{remark}

\begin{remark}\label{remark_conjecture_for_non_CM}
Let $A$ be a one-dimensional local ring. Then $A$ is Cohen-Macaulay if and only if its depth is $\ge 1$, which means precisely that $\frak m$ contains at least one non-zero-divisor (which is in fact equivalent to $A$ not having any embedded point). So, in the strange situation that $A$ is not Cohen-Macaulay, $A$ coincides with its own total quotient ring $Q(A)$ (:=$S^{-1}A$ where $S$ is the set of zero-divisors of $A$). Otherwise pick any non-zero divisor $t\in A$: then $V(t)=\{\frak m\}$ in $\Spec A$, and $Q(A)=A[t^{-1}]=\Frac A$. In other words, when $A$ is Cohen-Macaulay (e.g., reduced or a local complete intersection), $\Frac A=Q(A)$, and this is why we choose to use the friendly notation $\Frac$. However, in full generality, $\Frac A$ is the punctured disk $\Spec A\setminus\frak m$. In fact, the philosophy of the conjecture is the following:
\begin{quote}
The $K$-theory of a one-dimensional, complete local ring is determined on its punctured disk and on all infinitesimal thickenings of the closed point.
\end{quote}

We also remark that if the reader wishes only to treat Cohen-Macaulay rings, then R.~Thomason and T.~Trobaugh's \cite{Thomason1990} theory of localisation and their notion of an `isomorphism infinitely near' which we will use may be replaced respectively by the older approach to localisation for Cartier divisors, due to D.~Quillen and written down by D.~Grayson \cite{Grayson1976}, and $S$-analytic isomorphisms (e.g.~\cite{Weibel1980}; the key idea goes back to Karoubi \cite[App.~5]{Karoubi1974}).
\end{remark}

\begin{remark}\label{remark_homotopy_approach}
From a homotopy theoretic point of view one may prefer to work with \[K_n^\sub{top}(A):=\pi_n(\holim_r K(A/\frak m^r)),\] where $K$ denotes any functorial choice of the $K$-theory spectrum. This fits into a short exact sequence \[0\to {\projlim_r}^1 K_{n+1}(A/\frak m^r)\to K_n^\sub{top}(A)\to\hat K_n(A)\to 0,\] where the term on the left is the first right derived functor of $\projlim$. Therefore we would obtain a weaker conjecture if we were to replace $\hat K_n$ by $K_n^\sub{top}$. However, in each of the following two diametric cases, we will see that $\projlim_r^1 K_{n+1}(A/\frak m^r)=0$ and so $K_n^\sub{top}(A)\isoto\hat K_n(A)$:
\begin{enumerate}
\item $A$ a one-dimensional, reduced, excellent, local ring containing $\bb Q$ (theorem \ref{theorem_Ktop_is_Khat}).
\item $A$ a one-dimensional, local ring with finite residue field (indeed, we will note in proposition \ref{proposition_thanks_to_Vigleik} that $K_{n+1}(A/\frak m^r)$ is finite for all $r$, whence $\projlim_r^1 K_{n+1}(A/\frak m^r)=0$).
\end{enumerate}
Although the homotopy-theoretic groups $K_n^\sub{top}$ are more convenient for abstract functorial constructions, in this paper it does not matter whether we choose to work with $\hat K_n$ or $K_n^\sub{top}$. For another homotopy-theoretic approach, see corollary \ref{corollary_homotopy_description}.
\end{remark}

The first aim of this section is to develop tools to pass between Henselian and complete rings, and then we will state our main results.

\begin{proposition}\label{proposition_local_injectivity_from_Artin_Approximation}
Suppose that $A$ is a one-dimensional, excellent, Henselian local ring, and let $n\ge 0$. Then $K_n(A)\to K_n(\hat A)$ is injective.
\end{proposition}
\begin{proof}
This is a standard Artin Approximation type argument which applies to any functor $A\op{-}Algs\to Ab$ commuting with filtered inductive limits: the argument is as follows.

By Neron-Popescu disingularization \cite{Popescu1985} (see also \cite{Popescu1986} and \cite{Swan1998}), $\hat A$ may be written as an filtered inductive limit of finite-type, smooth $A$-algebras. Since $K$-theory commutes with filtered inductive limits, it is now enough to show that if $R$ is a finite-type, smooth $A$-algebra which admits an $A$-algebra morphism $f$ to $\hat A$, then $K_n(A)\into K_n(R)$. But the assumption on the existence of $f$ means that $R/f^{-1}(\frak m_{\hat A})=A/\frak m_A$; that is, $A/\frak m_A\to R\otimes_A A/\frak m_A$ has a section. Since $A$ is Henselian and $R$ is smooth over $A$, this lifts to a section of $A\to R$. Hence $K_n(A)\to K_n(R)$ has a section, and thus it is certainly injective.
\end{proof}

\begin{corollary}\label{corollary_equivalence_of_completed_and_Henselian}
Suppose that $A$ is a one-dimensional, excellent, Henselian local ring, and let $n\ge 0$. Then the following square of abelian groups is bicartesian and the vertical arrows are injective:
\[\xymatrix{
K_n(A)\ar[r]\ar@{^(->}[d] & K_n(\Frac A)\ar@{^(->}[d]\\
K_n(\hat A)\ar[r] & K_n(\Frac\hat A)
}\]
\end{corollary}
\begin{proof}
The natural map $A\to\hat A$ is an isomorphism infinitely near $\frak m$ (see the appendix for a review of this notion) and so there is a resulting long exact Mayer-Vietoris sequence: \[\cdots\to K_n(A)\to K_n(\hat A)\oplus K_n(\Frac A)\to K_n(\Frac\hat A)\to\cdots\] The previous proposition implies that this breaks into short exact sequences, from which everything follows.
\end{proof}

It follows from the corollary that if $A$ is a one-dimensional, excellent, Henselian local ring, then \[\ker(K_n(A)\to \hat K_n(A)\oplus K_n(\Frac A))=\ker(K_n(\hat A)\to\hat K_n(A)\oplus K_n(\Frac A)).\] So Conjecture 1 is equivalent to the seemingly stronger:
\begin{quote}
{\bf Conjecture 1':} If $A$ is a one-dimensional, excellent, Henselian local ring, then the diagonal map \[K_n(A)\to \hat K_n(A)\oplus K_n(\Frac A)\] is injective for all $n\ge 0$. 
\end{quote}

Next we show how the validity of the conjecture yields a long exact Mayer-Vietoris sequence in $K$-theory for $\hat K_n(A)$, analogous to the infinitely near one used in the proof of the previous corollary. Let $A$ be a one-dimensional, Noetherian local ring; we define the completed $K$-groups of $\Frac A$, denoted $\hat K_n(\Frac A)$, to be given by the following pushout diagram:
\xysquare{K_n(\hat A)}{K_n(\Frac\hat A)}{\hat K_n(A)}{\hat K_n(\Frac A)}{->}{->}{-->}{-->}
So conjecture 1 for $\hat A$ predicts that this diagram is not only cocartesian, but actually bicartesian.

\begin{remark}
Thanks to corollary \ref{corollary_equivalence_of_completed_and_Henselian}, we could equivalently define $\hat K_n(\Frac A)$ as the pushout of \xysquare{K_n(A^h)}{K_n(\Frac A^h)}{\hat K_n(A)}{}{->}{->}{}{} i.e.~We can replace completions by Henselisations everywhere.
\end{remark}

\begin{remark}
The problem of whether the group $\hat K_n(\Frac A)$ depends only on $\Frac A$, and not on $A$, is closely related to conjecture 1. But for this reason we will never write $\hat K_n(F)$, even if $F=\Frac A$.
\end{remark}

\begin{proposition}\label{proposition_MV_sequence_for_K_hat}
Suppose that $A$ is a one-dimensional, Noetherian local ring such that conjecture 1 holds for $\hat A$. Then there is a natural long exact Mayer-Vietoris sequence \[\cdots\to K_n(A)\to\hat K_n(A)\oplus K_n(\Frac A)\to\hat K_n(\Frac A)\to\cdots\]
\end{proposition}
\begin{proof}
Using the long exact Mayer-Vietoris sequence from the proof of corollary \ref{corollary_equivalence_of_completed_and_Henselian}, we may construct a commutative diagram
\[\xymatrix{
\cdots \ar[r] & K_n(A) \ar[d]\ar[r] & K_n(\hat A)\oplus K_n(\Frac A) \ar[r]\ar[d] & K_n(\Frac\hat A) \ar[r]\ar[d] & K_{n-1}(A) \ar[r]\ar[d] & \cdots\\
\cdots \ar@{-->}[r] & K_n(A) \ar[r] & \hat K_n(A)\oplus K_n(\Frac A) \ar[r] & \hat K_n(\Frac A)\ar@{-->}[r] & K_{n-1}(A) \ar[r] & \cdots
}\]
(with exact top row) where the dotted arrows are defined by the universal pushout property of the central square in such a way that the bottom row is a complex and the diagram commutes. A diagram chase shows that if conjecture 1 is true for $\hat A$ then the bottom row is actually exact.
\end{proof}

\begin{remark}
Continuing remark \ref{remark_homotopy_approach}, we explain a more homotopy theoretic alternative to $\hat K_n(\Frac A)$.

Let $K^\sub{top}(A):=\holim_rK(A/\frak m^r)$, so that $K_n^\sub{top}(A)=\pi_n(K^\sub{top}(A))$. Next let $K^\sub{top}(\Frac A)$ denote the homotopy pushout, in the category of spectra, of the diagram \xysquare{K(A)}{K(\Frac A)}{K^\sub{top}(A)}{K^\sub{top}(\Frac A),}{->}{->}{-->}{-->} and set $K_n^\sub{top}(\Frac A):=\pi_n(K_n^\sub{top}(\Frac A))$. In the category of spectra, homotopy pushout and pullback diagrams coincide, so there is a resulting long exact sequence \[\cdots\to K_n(A)\to K_n^\sub{top}(A)\oplus K_n(\Frac A)\to K_n^\sub{top}(\Frac A)\to\cdots \] The weaker version of conjecture 1 introduced in remark \ref{remark_homotopy_approach} predicts that this breaks into short exact sequences if $A$ is complete.

It may appear strange that $K_n^\sub{top}(\Frac A)$ was defined directly using $A$, whereas $\hat K_n(\Frac A)$ was defined via a pushout using the completion $\hat A$. However, the map $A\to\hat A$ is an isomorphism infinitely near $\frak m$, resulting in a homotopy cartesian square \xysquare{K(A)}{K(\Frac A)}{K(\hat A)}{K(\Frac\hat A)}{->}{->}{->}{->} of spectra, whence the natural map $K^\sub{top}(\Frac A)\to K^\sub{top}(\Frac\hat A)$ is a weak equivalence. That is, we are free to replace $A$ by $\hat A$ when defining $K_n^\sub{top}(\Frac A)$.
\end{remark}

Having discussed some theoretical issues surrounding the conjecture, we now turn to results. Initial faith in the conjecture was inspired by the following special case:

\begin{proposition}\label{proposition_seminormal}
Let $A$ be a one-dimensional, excellent, Henselian local ring. Then conjecture 1' is true for $K_2$ if $A$ contains a field and is seminormal with rational singularities.
\end{proposition}
\begin{proof}
By the arguments above, it is enough to treat the case that $A$ is actually complete. So, by the appendix reviewing seminormality, $A$ has the following structural description: it contains a coefficient field $k$, and ideals $I_1,\dots,I_n$, in such a way that $A\cong k\oplus I_1\oplus\dots\oplus I_n$ as an abelian group. Moreover, the maximal ideal of $A$ is $\frak m=I_1+\dots+I_n$, and each ring $k+I_j$, which is the localisation of $A$ away from the prime ideal $\frak q_j=\sum_{i\neq j}I_i$, is a complete discrete valuation ring with residue field $k$.

Applying Dennis-Krusemeyer's theorem on the $K$-theory of rings with such structure (see theorem \ref{theorem_Dennis_Krusemeyer}), we deduce that there is a resulting isomorphism \[K_2(A)\cong K_2(k)\oplus \bigoplus_{i=1}^n L_i\oplus \bigoplus_{i<j}(I_i/I_i^2\otimes_k I_j/I_j^2),\] where $L_i:=\ker(K_2(k+I_i)\to K_2(k))$.

Let $B=A/\frak m^2$, and notice that the structural description of $A$ induces a similar description of $B$: namely, $B\cong k\oplus I_1/I_1^2\oplus\dots\oplus I_n/I_n^2$. A second application of Dennis-Krusemeyer's theorem yields \[K_2(B)\cong K_2(k)\oplus \bigoplus_{i=1}^nL_i'\oplus\bigoplus_{i<j}(I_i/I_i^2\otimes_k I_j/I_j^2),\] where $L_i':=\ker(K_2(k+I_i/I_i^2)\to K_2(k))$.

Therefore, ignoring the $L_i$ and $L_i'$ factors, $K_2(A)\to K_2(B)$ is an isomorphism. But $k+I_i$ is a complete discrete valuation ring of equal characteristic, so its $K_2$ embeds into $K_2$ of its field of fractions, by Quillen's proof of the Gersten conjecture. In conclusion, $K_2(A)\to K_2(A/\frak m^2)\oplus K_2(\Frac A)$ is injective, which is more than enough to complete the proof.
\end{proof}

The following is our main theorem giving evidence for the conjecture:

\begin{theorem}\label{theorem_char_zero_local_result}
Let $A$ be a one-dimensional, excellent, Henselian local ring. Then conjecture 1' is true for all $n\ge 0$ if $A$ is reduced and contains $\bb Q$.
\end{theorem}
\begin{proof}
The proof is deferred until section \ref{section_main_proofs}.
\end{proof}

The nilpotent extension we can offer on top of the previous theorem is the following:

\begin{theorem}\label{theorem_char_zero_truncated_polynomials}
Let $A$ be a one-dimensional, excellent, Henselian local ring; as in the previous theorem, suppose that $A$ is reduced and contains $\bb Q$. Then conjecture 1' is also true for $A[t]/\pid{t^e}$ for all $n,e\ge 0$.
\end{theorem}
\begin{proof}
Also deferred until section \ref{section_main_proofs}.
\end{proof}

We finish this section by showing that a property akin to completeness is required for the conjecture to hold; even a mild singularity causes it to fail for local rings of curves. In fact, we prove a stronger result which shows moreover that the weaker $K_n^\sub{top}$ version of conjecture 1 discussed in remark \ref{remark_homotopy_approach} also fails in such a situation:

\begin{proposition}
Let $k$ be any field and let $A$ be the local ring of the singular point on the nodal curve $Y^2=X^2(X+1)$ over $k$. Then the map \[K_2(A)\to K_2(\hat A)\oplus K_2(\Frac A)\] is not injective.

(This implies that $A$ fails to satisfy the conjecture since there are natural maps $K_n(\hat A)\to K_n^\sub{top}(A)\to \hat K_n(A)$.)
\end{proposition}
\begin{proof}
$A$ is a one-dimensional, local domain, essentially of finite type over $k$, and it is the prototypical example of a seminormal ring with rational singularities. In fact, for most of the proof, we will work with any $A$ which is a one-dimensional, seminormal local ring, essentially of finite type over $k$, with rational singularities. Set $F=\Frac A$, and notice that $F$ is a finite product of fields, hence regular.

Since $A$ has rational singularities, corollary \ref{corollary_K_1_regularity_for_seminormal_rings} implies that $A$ is $K_1$-regular, and so there is an exact sequence (see the appendix for a discussion of $KV$-theory) \[NK_2(A)\to K_2(A)\to KV_2(A)\to 0.\] The same holds for $\hat A$ in place of $A$, and so we obtain two exact sequences: 
\[\xymatrix{
0\ar[r]& \op{Im}NK_2(A)\ar[r]\ar[d]&K_2(A)\ar[r]\ar[d]& KV_2(A)\ar[d]\ar[r]& 0\\
0\ar[r]& \op{Im}NK_2(\hat A)\ar[r]&K_2(\hat A)\ar[r]&KV_2(\hat A)\ar[r]&0
}\]
Moreover, since $NK_2(A)\to NK_2(\hat A)$ is an isomorphism \cite[Corol.~1.4]{Weibel1980}, the left vertical arrow is surjective. If it is not injective then the proof is finished, because $NK_2(A)$ vanishes in $K_2(F)$. Therefore we may assume it is injective, hence an isomorphism, in which case \[\ker(K_2(A)\to K_2(\hat A))=\ker(KV_2(A)\to KV_2(\hat A)).\] Moreover, since $K_2(F)=KV_2(F)$, we even deduce that \[\ker(K_2(A)\to K_2(\hat A)\oplus K_2(F))=\ker(KV_2(A)\to KV_2(\hat A)\oplus KV_2(F)).\] Even more, Quillen's proof of the Gersten conjecture in the geometric case implies that $KV_2(\tilde A)\to KV_2(F)$ is injective, so the kernels of the previous line, which we wish to show are non-zero, are the same as \[\kappa:=\ker(KV_2(A)\to KV_2(\hat A)\oplus KV_2(\tilde A)).\]

Next, lemma \ref{lemma_structure_of_complete_seminormal_rings} and the subsequent comment imply that we may write $\hat A\cong k\oplus I_1\oplus\cdots\oplus I_m$, where each ring $k+I_i$ is a complete discrete valuation ring, and our modification of theorem \ref{theorem_Dennis_Krusemeyer} for Karoubi-Villameyor theory then implies that \[KV_*(\hat A)\cong KV_*(k)\oplus\bigoplus_{i=1}^mKV_*(k+I_i,I_i).\] Since $\tilde{\hat A}=\prod_{i=1}^m k+I_i$, it is now clear that $KV_*(\hat A)\to KV_*(\tilde{\hat A})$ is injective. Therefore \[\kappa=\ker(KV_2(A)\to KV_2(\tilde A)).\]

Next, the $GL$-fibration
\[\xymatrix{
A\ar[r]\ar[d] & \tilde A\ar[d]\\
k\ar[r] & K:=\tilde A/\frak m_A
}\]
produces a long exact Mayer-Vietoris sequence in Karoubi-Villamayer theory \[\cdots\to KV_3(\tilde A)\oplus K_3(k)\to K_3(K)\to KV_2(A)\to KV_2(\tilde A)\oplus K_2(k)\to K_2(K)\to\cdots\] Since $K_*(k)$ is a direct summand of $KV_*(\tilde A)$, it follows from this sequence that
\begin{align*}
\kappa
	&=\ker(KV_2(A)\to KV_2(\tilde A)\oplus K_2(k))\\
	&=\Im(K_3(K)\to KV_2(A))\\
	&=K_3(K)/\Im(KV_3(\tilde A)\to K_3(K))
\end{align*}
In conclusion, to show that $\kappa$ is non-zero, it is necessary and sufficient to prove that $K_3(\tilde A)\to K_3(K)$ is not surjective!

So, finally, we let $A$ be the local ring of the nodal singularity as specified in the proposition. Then $B:=\tilde A$ is the semi-local ring obtained by localising $C:=k[t]$ away from two distinct points $x_1,x_2\in \bb A_k^1$. Quillen's localisation theorem implies that there is a short exact sequence \[0\to K_*(k)\to K_*(B)\to \bigoplus_{x\in\bb A_k^1\setminus\{x_1,x_2\}} K_{*-1}(k)\to 0.\] 

However, in the case $*=3$, the boundary map $\bigoplus_{x\neq x_1,x_2}\bor_x:K_3(B)\to\bigoplus_{x\neq x_1,x_2}K_2(k)$ is already surjective when restricted to the symbolic part $K_3^\sub{sym}(B)$ of $K_3(B)$; this is because $K_2(k)$ is generated by symbols and the tame symbols satisfy \[\bor_x\{\theta_1,\theta_2,t_y\}=\begin{cases}\{\theta_1,\theta_2\}&x=y\\0&x\neq y,\end{cases}\] if $x,y\in\bb A_k^1$, $\theta_i\in \mult k$, and $t_y\in k[t]$ is a local parameter at $y$.

Writing $K_3^\sub{ind}=K_3/K_3^\sub{sym}$ as usual, this implies that $K_3^\sub{ind}(k)\to K_3^\sub{ind}(B)$ is surjective. So, if \[K_3^\sub{ind}(B)\to K_3^\sub{ind}(B/\frak M_B)=K_3^\sub{ind}(k)\oplus K_3^\sub{ind}(k)\] were surjective (which would certainly follow from the surjectivity we are aiming to disprove), we would deduce that the diagonal map $K_3^\sub{ind}(k)\to K_3^\sub{ind}(k)\oplus K_3^\sub{ind}(k)$ were surjective. However, $K_3^\sub{ind}(k)$ is non-zero: for example, its $n$-torsion is $H^0(k,\mu_n^{\otimes 2})$ for any $n$ not divisible by $\op{char}k$ by \cite{Levine1987}, and this is non-zero by picking any such $n$ such that $\mu_n\subseteq\mult k$. This completes the proof.
\end{proof}

\section{Calculations in residue characteristic zero and proofs of theorems \ref{theorem_char_zero_local_result} and \ref{theorem_char_zero_truncated_polynomials}}\label{section_main_proofs}
In this section we prove conjecture 1' in the case when $A$ is reduced and contains $\bb Q$, or is truncated polynomials over such a ring, and simultaneously give additional structural results for the completed $K$-groups. The proofs are based on comparison theorems with cyclic homology, namely T.~Goodwillie's \cite{Goodwillie1986} result for nilpotent ideals, and G.~Corti\~nas' proof \cite{Cortinas2006} of the KABI conjecture.

\subsection{Proof of theorem \ref{theorem_char_zero_local_result}}
The essential technical results and ideas to prove the theorem come from two papers by A.~Krishna \cite{Krishna2005, Krishna2010} on Artin-Rees type properties in Hochschild and cyclic homology. Before describing these results and sketching our proof, we need to summarise the theory of categories of pro objects.

\begin{remark}\label{remark_pro_cats}
Everything we need about categories of pro objects may be found in one of the standard references, such as the appendix to \cite{ArtinMazur1969}, or \cite{Isaksen2002}. We will use $\op{Pro}Ab$, the category of pro abelian groups. We find this to be a convenient and conceptual way to state many of the results, replacing Krishna's repeated use of his `doubling trick'.

If $\cal C$ is a category, then $\op{Pro}\cal C$, the {\em category of pro objects of $\cal C$}, is the following: an object of $\op{Pro}\cal{C}$ is a contravariant functor $X:\cal I\to \cal{C}$, where $\cal I$ is a small cofiltered category (it is fine to assume that $\cal I$ is a codirected set); this object is usually denoted \[\projlimf_{i\in\cal I} X(i)\quad\quad\mbox{ or }\quad\quad\projlimf_{i}X(i),\] or by some other suggestive notation. The morphisms in $\op{Pro}\cal{C}$ from $\projlimf_{i\in\cal I}X$ to $\projlimf_{j\in \cal J}Y(j)$ are \[\Hom_{\op{Pro}\cal{C}}(\projlimf_{i\in\cal I}X,\projlimf_{j\in\cal J}Y):=\projlim_{j\in \cal J}\indlim_{i\in \cal I}\Hom_\cal{C}(X(i),Y(j)),\] where the right side is a genuine pro-ind limit in the category of sets. Composition is defined in the obvious way. 

There is a fully faithful embedding $\cal C\to\op{Pro}\cal C$. Assuming that projective limits exist in $\cal C$, there is a realisation functor \[\op{Pro}\cal C\to\cal C,\quad \projlimf_{i\in\cal I} X(i)\mapsto \projlim_{i\in\cal I} X(i),\] which is left exact but not right exact (its derived functors are precisely $\projlim^1,\projlim^2$, etc.), and which is a left adjoint to the aforementioned embedding.

Suppose that $\cal A$ is an abelian category. Then $\op{Pro}\cal A$ is an abelian category. Moreover, given a system of exact sequences \[\cdots\To X_{n-1}(i)\To X_n(i)\To X_{n+1}(i)\To\cdots,\] the formal limit \[\cdots \To \projlimf_{i\in\cal I}X_{n-1}(i)\To \projlimf_{i\in\cal I}X_n(i)\To \projlimf_{i\in\cal I}X_{n+1}(i)\To\cdots\] is an exact sequence in $\op{Pro}\cal A$. Of course, we cannot deduce that \[\cdots \To \projlim_{i\in\cal I}X_{n-1}(i)\To \projlim_{i\in\cal I}X_n(i)\To \projlim_{i\in\cal I}X_{n+1}(i)\To\cdots\] is exact in $\cal A$ (assuming that all these projective limits exist in $\cal A$) because the realisation functor is not right exact.
\end{remark}

Let $k$ be a field of characteristic zero and $\ell$ any subfield of $k$ (usually $\ell=\bb Q$ for us). Let $A$ be a reduced ring which is essentially of finite type over $k$, and $B=\tilde A$ its normalisation; assume that $B$ is smooth over $k$ (e.g., this is automatic if $A$ is one-dimensional, which is our only case of interest). A {\em conducting ideal} is any non-zero ideal of $B$ contained inside $A$.

Krishna proves, using a variety of Artin-Rees type, preliminary results concerning Andr\'e-Quillen, Hochschild, and cyclic homology, that if $I\subseteq A$ is any conducting ideal, then the natural map on double-relative cyclic homology \[HC_n^\ell(A,B,I^r)\to HC_n^\ell(A,B,I)\] is zero for $r\gg0$. Thanks to Corti\~nas' proof of the KABI conjecture we can replace $HC_n^{\bb Q}$ by $K_{n+1}$ in this result.

Two additional comments should be made at this point: Firstly, in \cite{Krishna2010}, Krishna works under the assumption that $A$ is actually a domain. However, all his proofs remain valid if $A$ is merely reduced: $B$ will then be a finite product of smooth domains, which is still smooth, and this is the important property he uses. Secondly, the required size of $s$ may depend on $n$, but this is never a problem.

In terms of pro abelian groups, this means that $\projlimf_{r\ge 1} K_n(A,B,I^r)=0$ in $\op{Pro}Ab$. Hence \[\projlimf_{r\ge 1} K_n(A,I^r)\to\projlimf_{r\ge 1} K_n(B,I^r)\] is an isomorphism, and applying the realisation functor tells us that $\projlim_{r\ge 1} K_n(A,I^r)\to\projlim_{r\ge 1} K_n(B,I^r)$ is an isomorphism of groups. Moreover, the usual splicing argument (in $\op{Pro}Ab$) gives us a long exact Mayer-Vietoris sequence in $\op{Pro}Ab$: \[\cdots K_n(A)\To \projlimf_{r\ge 1} K_n(A/I^r)\oplus K_n(B)\To \projlimf_{r\ge 1} K_n(B/I^r)\To\cdots,\tag{MV}\] which we will explain in a moment is a key component of our main proof.

If $A$ is one-dimensional and local, and $I$ is a fixed conducting ideal, then the following three systems of ideals are all mutually commensurable (i.e., cofinite in one another):\[\{\frak m_A^r:r\ge 1\},\quad \{I^r:r\ge 1\},\quad\{J:J\mbox{ a conducting ideal}\}\] So projective limits over over the three systems are the same in this case, and we will pass between them without mention.

We may now explain the main ideas of the proof of theorem \ref{theorem_char_zero_local_result}; let $A$ be a one-dimensional, reduced, excellent, Henselian local ring containing $\bb Q$, let $F=\Frac A$, and let $B=\tilde A$. To prove that $K_n(A)\to \hat K_n(A)\oplus K_n(F)$ is injective, it is enough to show that $K_n(A)\to \hat K_n(A)\oplus K_n(B)$ is injective, since $K_n(B)\into K_n(F)$ by Quillen's proof of the Gersten conjecture in this case. Since the realisation functor $\projlim:\op{Pro}Ab\to Ab$ is left exact, it is now enough to show that $K_n(A)\to\projlimf_rK_n(A/\frak m^r)\oplus K_n(B)$ is injective in $\op{Pro}Ab$. But the sequence (MV) above reduces this in turn to checking the surjectivity of $\projlimf_rK_n(A/\frak m^r)\oplus K_n(B)\to \projlimf_rK_n(B/\frak M^r)$ in $\op{Pro}Ab$; this surjectivity is essentially the content of lemma \ref{lemma_not_top_degree_of_Hodge_for_truncated_polys} -- remark \ref{remark_surjectivity_reinterpretation}. Unfortunately, the sequence (MV) has only been established for rings essentially of finite type over a field, which $A$ is not; this difficulty is overcome by reducing to the case that $A$ is the Henselization of such an essentially finite-type ring and then passing to the limit. It is easier to work with relative $K$-groups and treat the residue fields separately.

Now we begin applying and modifying Krishna's results for our purposes. The essence of $K_2$-versions of some of the following results are contained in the proof of \cite[Lem.~3.3]{Krishna2005}. We will see that the top degree part of the Hodge/Adams decomposition behaves completely differently to the lower degree parts, and we must treat them separately. We begin with the lower degree parts:

\begin{lemma}\label{lemma_not_top_degree_of_Hodge_for_truncated_polys}
Let $K/k$ be an arbitrary extension of fields of characteristic $0$, and let $e\ge 0$. For any $1\le i<n$, the natural map \[\tilde{HC}_n^{(i)}(K[t]/\pid{t^{2e}})\to \tilde{HC}_n^{(i)}(K[t]/\pid{t^e})\] of reduced cyclic homology groups (with respect to $k$) is zero.
\end{lemma}
\begin{proof}
For the sake of brevity, write $B_r=k[t]/\pid{t^r}$ for any $r$. All Hochschild and cyclic homologies in this proof are taken with respect to $k$, so the superscript $k$ will be omitted. The proof will work whenever $K$ is a Noetherian ring which is geometrically regular over $k$, for then $HH_*(K)\cong\Omega_{K/k}^*$, which is all we need $K$ to satisfy.

For any $k$-algebra $A$, we write \[HH_n^{(<n)}(A)=HH_n(A)/HH_n^{(n)}(A)=\bigoplus_{i=1}^{n-1}HH_n^{(i)}(A)\] for the quotient of $HH_n(A)$ by the top degree part of its Hodge decomposition. If $A=A_0\oplus A_1\oplus\cdots$ is positively graded then we write $\tilde{HH}_n(A)=HH_n(A)/HH_n(A_0)$ for the reduced Hochschild homology. The obvious conjunction of these notations will be also be used. Our claim in the statement of the lemma is that \[\tilde{HC}_n^{(<n)}(B_{2e}\otimes_k K)\To \tilde{HC}_n^{(<n)}(B_e\otimes_k K)\] is zero.

Let $A\to A'$ be a morphism of positively graded $k$-algebras; we will prove that the following conditions are equivalent:
\begin{enumerate}
\item $\tilde{HC}_n^{(<n)}(A_K)\to\tilde{HC}_n^{(<n)}(A_K')$ is zero for all $n\ge 0$.
\item As (i), but replacing $HC$ by $HH$.
\item As (ii), but also replacing $A_K$ and $A'_K$ by $A$ and $A'$ respectively.
\end{enumerate}

Firstly, the SBI sequence for the reduced Hochschild and cyclic homology of $A$ breaks into short exact sequences \cite[Thm.~4.1.13]{Loday1992}; moreover, the $S$, $B$, and $I$ maps respect the Hodge grading in such a way that we may ignore the top degree [Prop.~4.6.9, op.~cit.,]: \[0\to\tilde{HC}_{n-1}^{(<n-1)}(A_K)\xto{B} \tilde{HH}_n^{(<n)}(A_K)\xto{I}\tilde{HC}_n^{(<n)}(A_K)\to 0\] The equivalence (i)$\Leftrightarrow$(ii) follows from a trivial induction (to start the induction notice that $HH_i^{(<i)}=0$ for $i=0,1$).

Next, by the Eilenberg-Zilber theorem, \[\tilde{HH}_n(A_K)=\bigoplus_{p+q=n}\tilde{HH}_p(A)\otimes_kHH_q(K),\] and this decomposition is known to be compatible with the Hodge decompositions in that $\tilde{HH}_p^{(i)}(A)\otimes_k HH_q^{(j)}(K)\subseteq \tilde{HH}_n^{(i+j)}(A_K)$ [Prop.~4.5.14, op.~cit.]. But $K$ is a limit of finitely generated separable field extensions of $k$, so $HH_q(K)=HH_q^{(q)}(K)$ for all $q\ge 0$. Therefore \[\tilde{HH}_n^{(<n)}(A_K)=\bigoplus_{p+q=n}\tilde{HH}_p^{(<p)}(A)\otimes_kHH_q(K)\] It is now evident that $\tilde{HH}_n^{(<n)}(A_K)\to \tilde{HH}_n^{(<n)}(A_K')$ is zero for all $n\ge 0$ if and only if $\tilde{HH}_n^{(<n)}(A)\to \tilde{HH}_n^{(<n)}(A)$ is zero for all $n$, which is precisely the equivalence (ii)$\Leftrightarrow$(iii).

So, we have reduced the lemma to the case $K=k$, i.e.~to proving that \[\tilde{HH}_n^{(<n)}(B_{2e})\To \tilde{HH}_n^{(<n)}(B_e)\] is zero for all $n\ge 0$. This seems to be a reasonably well-known result which follows from filtration arguments; in any case it follows from \cite[Lem.~2.2]{Krishna2005}.
\end{proof}

In the following corollary and subsequent lemma notice that $K_n(B/I,\frak M/I)$ is a relative $K$-group for a nilpotent ideal in a $\bb Q$-algebra, hence is a $\bb Q$-vector space by C.~Weibel \cite[1.5]{Weibel1982}; so $K_n(B/I,\frak M/I)=K_n(B/I,\frak M/I)_{\bb Q}$. Of course, the same applies replacing $I$ by $I^2$.

\begin{corollary}\label{corollary_not_top_degree_of_Adams}
Let $B$ be a normal, one-dimensional, reduced semi-local ring containing $\bb Q$. Let $\frak M$ denote its Jacobson radical and let $I\subset B$ be an ideal with radical $\frak M$. For any $1\le i<n$, the natural map \[K_n^{(i)}(B/I^2,\frak M/I^2)\to K_n^{(i)}(B/I,\frak M/I)\] is zero.
\end{corollary}
\begin{proof}
The Goodwillie isomorphism $K_n(B/I,\frak M/I)\isoto HC_{n-1}(B/I,\frak M/I)$ respects the Adams/Hodge decompositions by \cite{Cathelineau1990}, thus inducing \[K_n^{(i)}(B/I,\frak M/I)\isoto HC_{n-1}^{(i-1)}(B/I,\frak M/I)\] and similarly for $I^2$ in place of $I$ (here $HC=HC^{\bb Q}$).

Next notice that $B/I\cong\prod_{\frak n}B_{\frak n}/IB_{\frak n}$, where $\frak n$ varies over the finitely many maximal ideals of $B$; for each $\frak n$, there are compatible isomorphisms $B_{\frak n}/IB_{\frak n}\cong K[t]/\pid{t^e}$, $B_{\frak n}/I^2B_{\frak n}\cong K[t]/\pid{t^{2e}}$, for some integer $e>0$ and some characteristic zero field $K$ (both depending on $\frak n$). Therefore $HC_{n-1}^{(i)}(B/I,\frak M/I)$ is a finite direct sum of terms of the form \[HC_{n-1}^{(i-1)}(K[t]/\pid{t^e},tK[t]/\pid{t^e}) = \tilde{HC}_{n-1}^{(i-1)}(K[t]/\pid{t^e}),\] and similarly for $I^2$.

We have reduced the problem to proving that \[\tilde{HC}_{n-1}^{(i-1)}(K[t]/\pid{t^{2e}})\to\tilde{HC}_{n-1}^{(i-1)}(K[t]/\pid{t^e})\] is zero, which is exactly the previous lemma.
\end{proof}

Next we analyse the top degree part of the Adams decomposition:

\begin{lemma}\label{lemma_top_degree_of_Adams}
Let $B$ be a semi-local ring containing $\bb Q$. Let $\frak M$ denote its Jacobson radical and let $I\subset B$ be an ideal with radical $\frak M$. Then \[K_n^{(n)}(B/I,\frak M/I)\subseteq\Im(K_n(B,\frak M)\To K_n(B/I,\frak M/I)).\]
\end{lemma}
\begin{proof}
Notice that $B/I$ is a finite product of Artinian local rings of residue characteristic zero, so $B/I\to B/\frak M$ splits and therefore \[K_n^{(n)}(B/I,\frak M/I)=\ker(K_n^{(n)}(B/I)_{\bb Q}\to K_n^{(n)}(B/\frak M)_{\bb Q}).\]

Now we use the following classical Nesterenko-Suslin result \cite{Suslin1989}: if $R$ is a local ring with infinite residue fields, then $K_n^{(n)}(R)_{\bb Q}\cong K_n^M(R)_{\bb Q}$. Although $B/I$ and $B/\frak M$ are not local rings, they are products of local rings and so Nesterenko-Suslin's result clearly remains valid. In conclusion, \[K_n^{(n)}(B/I,\frak M/I)=\ker(K_n^M(B/I)\to K_n^M(B/\frak M))\otimes_{\bb Z}\bb Q.\]

Next there is the following standard result concerning Milnor $K$-theory: If $R$ is a ring and $J\subseteq R$ is an ideal contained inside its Jacobson radical, then the kernel of $K_n^M(R)\to K_n^M(R/J)$ is generated by Steinberg symbols of the form $\{a_1,a_2\dots,a_n\}$, where $a_1\in1+J$ and $a_2,\dots,a_n\in\mult A$. Indeed, if we let $\Lambda$ denote the subgroup of $K_n^M(R)$ generated by such elements, then it is enough to check that \[K_n^M(R/J)\to K_n^M(R)/\Lambda,\quad\{a_1,\dots,a_n\}\mapsto \{\tilde a_1,\dots,\tilde a_n\}\] is well-defined, where $\tilde a\in\mult R$ denotes an arbitrary lift of $a\in\mult{(R/J)}$.

Applying this with $R=B/I$ and $J=\frak M/I$, and noticing that $1+\frak M/I$ is a divisible subgroup of $\mult{(B/I)}$, we have proved that $K_n^{(n)}(B/I,\frak M/I)$ is generated by Steinberg symbols of the form $\xi=\{a_1,a_2,\dots,a_n\}$, where $a_1\in1+\frak M/I$ and $a_2,\dots,a_n\in \mult{(B/I)}$. Fix such an element $\xi$, and let $\tilde\xi=\{\tilde a_1,\dots,\tilde a_n\}$ be a Steinberg symbol in $K_n(B)$ obtained by taking arbitrary lifts of $a_1,\dots,a_n$ to $\mult B$; then $\tilde\xi\in\ker(K_n(B)\to K_n(B/\frak M)$. Since \[K_n(B,\frak M)\to\ker(K_n(B)\to K_n(B/\frak M))\] is surjective, we may further lift $\tilde\xi$ to $K_n(B,\frak M)$, and this proves that $K_n(B,\frak M)\to K_n(B/I,\frak M/I)$ covers $K_n^{(n)}(B/I,\frak M/I)$.
\end{proof}

\begin{corollary}
Let $B$ be a normal, one-dimensional, reduced semi-local ring containing $\bb Q$. Let $\frak M$ denote its Jacobson radical and let $I\subset B$ be an ideal with radical $\frak M$. Then \[K_n^{(n)}(B/I,\frak M/I)=\Im(K_n(B,\frak M)\To K_n(B/I,\frak M/I))=\Im(K_n(B/I^r,\frak M/I^r)\To K_n(B/I,\frak M/I))\] for any $r\ge 2$.
\end{corollary}
\begin{proof}
Immediate from corollary \ref{corollary_not_top_degree_of_Adams} and lemma \ref{lemma_top_degree_of_Adams} (applied to both $B$ and $B/I^r$).
\end{proof}

\begin{remark}\label{remark_surjectivity_reinterpretation}
In terms of the pro group language, the previous corollary implies that \[K_n(B,\frak M)\To\projlimf_rK_n(B/\frak M^r,\frak M/\frak M^r)\] is surjective.

The imminent proof of proposition \ref{proposition_ess_finite_type_case} will show that this is also true if $B$ is any one-dimensional, reduced local ring which is essentially of finite type over a field of characteristic zero.
\end{remark}

\begin{remark}
Our main results will follow once we take advantage of a long exact sequence involving relative $K$-groups, which we explain in this remark.

Suppose that $I\subseteq J\subseteq R$ are ideals in a ring, and that $R/I\to R/J$ splits. Then there is a long exact sequence of relative $K$-groups \[\cdots \to K_n(R,I)\to K_n(R/J)\to K_n(R/I,J/I)\to \cdots\tag{\dag},\] constructed as follows. Consider the long exact sequences for $R\to R/I$, $R\to R/J$, and $R/I\to R/J$:
\[\xymatrix{
\cdots\ar[r] & K_n(R,I)\ar[r]\ar[d] & K_n(R) \ar[r]\ar@{=}[d] & K_n(R/I)\ar[r]^{\bor}\ar[d] & \cdots\\
\cdots\ar[r] & K_n(R,J)\ar[r]\ar[d] & K_n(R) \ar[r]\ar[d] & K_n(R/J)\ar[r]\ar@{=}[d] & \cdots\\
\cdots\ar[r] & K_n(R/I,J/I)\ar[r]^\al & K_n(R/I) \ar[r] & K_n(R/J)\ar[r] & \cdots
}\]
Note that the lowest sequence breaks into short exact sequences since $R/I\to R/J$ has a section. From this it is easy to construct (\dag): the boundary map $K_n(R/I,J/I)\to K_{n-1}(R,I)$ is given by $\bor\circ\al$.
\end{remark}

In the case $n=2$, the following proposition is one of the main results of \cite{Krishna2005}:

\begin{proposition}\label{proposition_ess_finite_type_case}
Let $A$ be a one-dimensional, reduced local ring, essentially of finite type over a field of characteristic zero. Then the natural map \[K_n(A,\frak m)\to K_n(A/\frak m^r,\frak m/\frak m^r)\oplus  K_n(\tilde A,\frak M)\] is injective for $r\gg 0$.
\end{proposition}
\begin{proof}
We apply the previous remark to $I\subseteq\frak m\subseteq A$ and $I\subseteq\frak M\subseteq B=\tilde A$ and compare the resulting long exact sequences:
\[\xymatrix{
\cdots \ar[r] & K_{n+1}(A/I,\frak m/I) \ar[r]\ar[d]^{(2)} & K_n(A,I) \ar[r]\ar[d] & K_n(A,\frak m) \ar[r]\ar[d] & K_n(A/I,\frak m/I) \ar[r]\ar[d] & \cdots\\
\cdots \ar[r] & K_{n+1}(B/I,\frak M/I) \ar[r] & K_n(B,I) \ar[r] & K_n(B,\frak M) \ar[r]^{(1)} & K_n(B/I,\frak M/I) \ar[r] & \cdots
}\]
We take the formal limit over all conducting ideals and make two observations: firstly, remark \ref{remark_surjectivity_reinterpretation} tells us that arrow (1) become surjective in the limit; and secondly, arrow (2) becomes an isomorphism in the limit thanks to our $\op{Pro}Ab$ interpretation of Krishna's result on double relative $K$-groups (explained after remark \ref{remark_pro_cats}):
\[\xymatrix@C=5mm{
\cdots\ar[r] & \projlimf_IK_{n+1}(A/I,\frak m/I) \ar[r]\ar[d] & \projlimf_IK_n(A,I) \ar[r]\ar[d]^\cong & K_n(A,\frak m) \ar[r]\ar[d] & \projlimf_IK_n(A/I,\frak m/I)\ar[d]\ar[r]&\cdots\\
\cdots\ar@{->>}[r] & \projlimf_IK_{n+1}(B/I,\frak M/I) \ar[r] & \projlimf_IK_n(B,I) \ar[r] & K_n(B,\frak M) \ar@{->>}[r] & \projlimf_IK_n(B/I,\frak M/I)\ar[r]&\cdots
}\]
A simple diagram chase shows that $K_n(A,\frak m)\to\projlimf_I K_n(A/I,\frak m/I)$ is also surjective. So the top and bottom rows break into short exact sequences
\[\xymatrix{
0 \ar[r] & \projlimf_IK_n(A,I) \ar[r]\ar[d]^\cong & K_n(A,\frak m) \ar[r]\ar[d] & \projlimf_IK_n(A/I,\frak m/I) \ar[r]\ar[d] & 0\\
0 \ar[r] & \projlimf_IK_n(B,I) \ar[r] & K_n(B,\frak M) \ar[r] & \projlimf_IK_n(B/I,\frak M/I) \ar[r] & 0
}\]
whence the right square in this diagram is bicartesian.

Hence $K_n(A,\frak m)\to K_n(B,\frak M)\oplus K_n(A/I,\frak m/I)$ is injective for all sufficiently small conducting ideals $I$, which completes the proof since any power of $\frak m$ contains a conducting ideal.
\end{proof}

The proposition leads to an interesting refinement of the surjectivity discussed in remark \ref{remark_surjectivity_reinterpretation} away from the top degree part of the Adams decomposition:

\begin{corollary}\label{corollary_explicit decomposition}
With notation as in the proposition, the natural map \[K_n^{(i)}(A,\frak m)_{\bb Q}\to\projlim_{r\ge 0}K_n^{(i)}(A/\frak m^r,\frak m/\frak m^r)\oplus K_n^{(i)}(\tilde A,\frak M)_{\bb Q}\] is an isomorphism, for $0\le i<n$.
\end{corollary}
\begin{proof}
According to corollary \ref{corollary_not_top_degree_of_Adams}, \[\projlimf_I K_n^{(i)}(B/I,\frak M/I)=0,\] where the formal projective limit is taken over the set of all conducting ideals. Repeating the previous proof with $K_n$ replaced by $K_n^{(i)}\otimes\bb Q$, we arrive at a commutative diagram where the top row is short exact:
\[\xymatrix{
0 \ar[r] & \projlimf_IK_n^{(i)}(A,I)_{\bb Q} \ar[r]\ar[d]^\cong & K_n^{(i)}(A,\frak m)_{\bb Q} \ar[r]\ar[d] & \projlimf_IK_n^{(i)}(A/I,\frak m/I) \ar[r]\ar[d] & 0\\
0 \ar[r] & \projlimf_IK_n^{(i)}(B,I)_{\bb Q} \ar[r]^\cong & K_n^{(i)}(B,\frak M)_{\bb Q} \ar[r] & 0 \ar[r] & 0
}\]
It follows that once that \[K_n^{(i)}(A,\frak m)_{\bb Q}\to\projlimf_IK_n^{(i)}(A/I,\frak m/I)\oplus K_n^{(i)}(B,\frak M)_{\bb Q}\] is an isomorphism, from which the stated claim follows by taking the realisation functor $\projlim$ and replacing $\projlim_I$ by $\projlim_{\frak m^r}$.
\end{proof}

The proposition also allows us to prove a special case of the main theorem (from which the main theorem itself will easily follow):

\begin{theorem}\label{theorem_hens_of_ess_finite_type_case}
Let $A$ be a one-dimensional, reduced local ring, essentially of finite type over a field of characteristic zero. Then \[K_n(A^h)\to \hat K_n(A^h)\oplus K_n(\tilde{A^h})\] is injective.
\end{theorem}
\begin{proof}
Let $\frak m$ denote the maximal ideal of $A$ and $k=A/\frak m$ its residue field. Let $A'$ be a finite, local \'etale extension of $A$ with residue field $k$; then $A'$ satisfies the conditions of proposition \ref{proposition_ess_finite_type_case}, and has maximal ideal $\frak mA'$ and normalisation $\tilde{A'}=A'\otimes_A\tilde A$, so \[K_n(A',\frak mA')\to K_n(A'/\frak m^rA',\frak mA'/\frak m^rA')\oplus K_n(A'\otimes_A\tilde A,\frak M')\] is injective for $r\gg0$, where $\frak M'$ is the Jacobson radical of $\tilde{A'}$. But $A'/\frak m^rA'=A/\frak m^r$ and $\frak mA'/\frak m^rA'=\frak m/\frak m^r$, so the central term is $K_n(A/\frak m^n,\frak m/\frak m^n)$; moreover, the quotient map $A/\frak m^r\to k$ is split, so $K_n(A/\frak m^r,\frak m/\frak m^r)\to K_n(A/\frak m^r)$ is (split) injective. In conclusion, \[K_n(A',\frak mA')\to\hat K_n(A)\oplus K_n(A'\otimes_A\tilde A,\frak M')\] is injective.

Since $A^h=\indlim A'$, with $A'$ running over all finite, local \'etale extensions of $A$ with residue field $k$, we may pass to the limit to deduce that \[K_n(A^h,\frak mA^h)\to\hat K_n(A)\oplus K_n(\tilde{A^h},\frak M^h)\tag{\dag}\] is injective, where we write $\frak M^h$ for the Jacobson radical of $\tilde{A^h}$. 

Next notice that $\tilde{A^h}$ is a Henselian, one-dimensional, reduced, normal local ring; hence it is a finite product of Henselian discrete valuation rings, and therefore the quotient map $\tilde{A^h}\to\tilde{A^h}/\frak M^h$ actually splits. So $K_*(\tilde{A^h},\frak M^h)\to K_*(\tilde{A^h})$ is (split) injective and thus we arrive at a commutative diagram with exact rows:
\[\xymatrix{
0\ar[r] & K_n(A^h,\frak m A^h)\ar[r]\ar[d] & K_n(A^h)\ar[r]\ar[d] & K_n(k)\ar[d] \ar[r] & 0 \\
0\ar[r] & K_n(\tilde{A^h},\frak M^h)\ar[r] & K_n(\tilde{A^h})\ar[r] & K_n(\tilde{A^h}/\frak M^h) \ar[r] & 0
}\]
A quick diagram chase using this commutative diagram and the injectivity of (\dag) reveals that if $\xi\in K_n(A^h)$ dies in both $\hat K_n (A)$ (whence it dies in $K_n(k)$) and $K_n(\tilde{A^h})$, then $\xi=0$. This completes the proof.
\end{proof}

The aim of this section, namely the proof of theorem \ref{theorem_char_zero_local_result}, easily follows from the previous theorem using standard manipulations:

\begin{proof}[Proof of theorem \ref{theorem_char_zero_local_result}]
Let $A$ be a one-dimensional, reduced, excellent, Henselian local ring containing $\bb Q$. Let $k$ be the residue field of $A$. After picking a coefficient field for $\hat A$ and generators for its maximal ideal, it is clear how to construct a one-dimensional, local subring $A_\circ\subseteq \hat A$ which is essentially of finite-type over $k$ and which satisfies $\hat{A_\circ}=\hat A$. The ring $A_\circ$ is reduced because $\hat A$ is (since $A$ is excellent).

According to corollary \ref{corollary_equivalence_of_completed_and_Henselian} and the subsequent comments, conjecture $1'$ is equivalent for the rings $A,\hat A=\hat{A_\circ}$, and $A_\circ^h$. The previous theorem proves the conjecture for $A_\circ^h$, and this completes the proof.
\end{proof}

We finish this section by noting another application of the type of surjectivity which appeared in remark \ref{remark_surjectivity_reinterpretation}:

\begin{theorem}\label{theorem_Ktop_is_Khat}
Let $A$ be a one-dimensional, reduced, excellent, local ring containing $\bb Q$. Then the natural map \[K_n^\sub{top}(A)\to\hat K_n(A)\] is an isomorphism (see remark \ref{remark_homotopy_approach}).
\end{theorem}
\begin{proof}
Since $K_n^\sub{top}(A)$ and $\hat K_n(A)$ depends only on the quotients $A/\frak m^r$, for $r\ge 1$, we may argue as in the proof of theorem \ref{theorem_char_zero_local_result} above and replace $A$ first by its completion and then by a smaller subring; this reduces the problem to the case when $A$ contains a coefficient field $k$ over which it is essentially of finite type. In the remainder of the proof we work with such a ring $A$.

We noted during the proof of proposition \ref{proposition_ess_finite_type_case} that \[K_n(A,\frak m)\to \projlimf_r K_n(A/\frak m^r,\frak m/\frak m^r)\tag{\dag}\] is surjective in $\op{Pro}Ab$. Applying the realisation functor $\projlim:\op{Pro}Ab\to Ab$ yields a short exact sequence \[\dots \to{\projlim_r}^1K_n(A,\frak m)\to {\projlim_r}^1 K_n(A/\frak m^r,\frak m/\frak m^r)\to{{\projlim_r}^2\kappa_r}\to\dots,\] where $\kappa_r=\ker(K_n(A,\frak m)\to K_n(A/\frak m^r,\frak m/\frak m^r))$. But it is well-known that $\projlim^2$ of a countable system of abelian groups automatically vanishes, and $\projlim^1$ of a constant system of groups certainly vanishes, so this shows that \[{\projlim_r}^1 K_n(A/\frak m^r,\frak m/\frak m^r)=0.\]

Applying the same type of argument to the exact sequence \[0\to \projlimf_rK_n(A/\frak m^r,\frak m/\frak m^r)\to\projlimf_r K_n(A/\frak m^r)\to K_n(k)\to 0\] shows that $\projlim_r^1 K_n(A/\frak m^r)=0$, whence the result follows.
\end{proof}

\subsection{Proof of theorem \ref{theorem_char_zero_truncated_polynomials}}
We would like to extend theorem \ref{theorem_char_zero_local_result} to the non-reduced case by taking advantage of the Goodwillie isomorphism and then studying the relative cyclic homology of nilpotent ideals. Unfortunately, at present we can only handle truncated polynomial rings.

\begin{lemma}
Let $A$ be a one-dimensional, reduced local ring, essentially of finite type over a field of characteristic zero, and let $C=A[t]/\pid{t^e}$. Then \[\tilde{HC}_n(C)\to\tilde{HC}_n(C/\frak m_C^r)\oplus\tilde{HC}_n(\Frac C)\] is injective for $r\gg 0$, where the cyclic homologies are taken with respect to any fixed subfield $k\subseteq A$.
\end{lemma}
\begin{proof}
For a moment let $R=R_0\oplus R_1\oplus\cdots$ be a graded $k$-algebra, and $A$ an arbitrary $k$-algebra. Combining C.~Kassel's formula \cite{Kassel1987} for the cyclic homology of a graded algebra with Goodwillie's result \cite{Goodwillie1985} that the reduced SBI sequence splits into short exact sequences, one can deduce that \[\tilde{HC}_n(R\otimes_kA)\cong\bigoplus_{p+q=n}\tilde{HC}_p(R)\otimes_k HH_q(A)\] For details, see \cite{Weibel1989}, where C.~Weibel, L.~Reid, and S.~Geller also point out that this decomposition is not natural, but depends on choices of splitting of the short exact sequences of $k$-modules \[0\to \tilde{HC}_{n-1}(R)\to\tilde{HH}_n(R)\to\tilde{HC}_n(R)\to 0.\] Without these choices of splitting, one has only a decreasing filtration on $\tilde{HC}_n(R\otimes_kA)$ for which $\op{gr}^p\tilde{HC}_n(R\otimes_kA)\cong\tilde{HC}_p(R)\otimes_k HH_{n-p}(A)$.

Put $R=k[t]/\pid{t^e}$ and let $A$ be as in the statement of the proposition. Then according to one of Krishna's Artin-Rees results \cite[Corol.~6.2(iii)]{Krishna2010}, $HH_q(A)\to HH_q(A/\frak m^r)\oplus HH_q(\tilde A)$ is injective for $r\gg 0$. Since $\tilde A$ is smooth over $k$, we have $HH_q(\tilde A)=\Omega_{\tilde A/k}^q$, which embeds into $HH_q(\Frac A)=\Omega_{\Frac A/k}^q=\Omega_{\tilde A/k}^q\otimes_{\tilde A}\Frac A$.

So, let $r$ be large enough so that $HH_q(A)\to HH_q(A/\frak m^r)\oplus HH_q(\Frac A)$ is injective for $q=0,\dots,n$. Applying the Kassel-Goodwillie decomposition to the rings \[C=A[t]/\pid{t^e},\quad\Frac C=(\Frac A)[t]/\pid{t^e},\quad C/\frak m_A^rC=(A/\frak m_A^r)[t]/\pid{t^e},\] we deduce that \[\tilde{HC}_n(C)\to \tilde{HC}_n(C/\frak m_A^rC)\oplus \tilde{HC}_n(\Frac C)\] is injective for $r\gg 0$. To obtain the exact statement of the proposition, just notice that $\frak m_A^rC\subseteq\frak m_C^{r+e}$ for all $r\ge 1$.
\end{proof}

Theorem \ref{theorem_char_zero_truncated_polynomials} easily follows:

\begin{proof}[Proof of theorem \ref{theorem_char_zero_truncated_polynomials}]
Exactly as in the proof of theorem \ref{theorem_char_zero_local_result} we may reduce to the case when $A$ is the Henselization of a one-dimensional, reduced local ring which is essentially of finite type over a field of characteristic $0$.

For any ring $R$ containing $\bb Q$ the Goodwillie isomorphism implies there is a split exact sequence \[0\to \tilde{HC}_{n-1}(R[t]/\pid{t^e})\to K_n(R[t]/\pid{t^e})\to K_n(R)\to 0\] The proof is completed by realising $A$ as a filtered direct limit of rings $A'$ to which the previous lemma applies (just as in the proof of theorem \ref{theorem_hens_of_ess_finite_type_case}), passing to the limit, and of course using theorem \ref{theorem_char_zero_local_result}.
\end{proof}

\begin{remark}
More generally, the proof presented above shows that conjecture 1' is true whenever $A=A_0\otimes_{\bb Q}R$, where $A_0$ is a one-dimensional, reduced, excellent, Henselian local rings containing $\bb Q$, and $R$ is a graded Artinian $\bb Q$-algebra with $R_0=\bb Q$.
\end{remark}

\section{Some examples of and miscellaneous results concerning completed $K$-groups}\label{section_example}
In this section we give some examples of and structural results on $\hat K_n(A)$ and $\hat K_n(\Frac A)$ in two important cases: rings with finite residue field, and $\bb Q$-algebras.

\subsubsection*{Finite residue field}
We begin by focussing on the case when $A$ has a finite residue field. We first show $\hat K_n(A)$ is a profinite group, then offer two homotopy-theoretic interpretations of it, and then explicitly consider the case of a complete discrete valuation ring.

\begin{proposition}\label{proposition_thanks_to_Vigleik}
Let $A$ be a one-dimensional, Noetherian local ring, with finite residue field. Then $K_n(A/\frak m^r)$ is finite for all $n,r\ge 1$, and so $\hat K_n(A)$ is a profinite group.
\end{proposition}
\begin{proof}
Indeed, it seems to be a folklore result that the $K$-groups of a finite ring $R$ are themselves finite; I am grateful to V.~Angeltveit for explaining the argument to me. Firstly, Bass stability implies that, for any fixed $n$, $H_n(BGL(R)^+,\bb Z)=H_n(GL(R),\bb Z)=H_n(GL_m(R),\bb Z)$ for $m$ sufficiently large, and $H_n(GL_m(R),\bb Z)$ is finite for $n\ge 1$ since $GL_m(R)$ is a finite group. Thus all the integral homology groups of degree $\ge 1$ of the $K$-theory space $BGL(R)^+$ are finite.

Since $BGL(R)^+$ is an infinite loop space, its $\pi_1$ acts trivially on its $\pi_n$ for all $n\ge 1$, so the theory of Serre classes tells us that \[\pi_n(BGL(R)^+)\mbox{ is finite for all }n\ge 1\Longleftrightarrow H_n(BGL(R)^+,\bb Z)\mbox{ is finite for all }n\ge 1,\] completing the proof.
\end{proof}

The proposition has some important homotopy-theoretic consequences, continuing the theme of remark \ref{remark_homotopy_approach}:

\begin{corollary}\label{corollary_homotopy_description}
Let $A$ be a one-dimensional, Noetherian local ring with finite residue field. Then $K^\sub{top}_n(A)\to K_n(A)$ is an isomorphism for each $n\ge 0$.

If moreover $A$ has characteristic zero and is complete, then \[K_n(A)\cong \pi_n(K(A)^\comp)\] for $n>0$, where $K(A)^\comp$ denotes the profinite completion of the $K$-theory spectrum of $A$.
\end{corollary}
\begin{proof}
The first claim was already explained in remark \ref{remark_homotopy_approach}.

For the second claim, Cohen structure theory implies that $A$ is a finite $\bb Z_p$-algebra, and so \cite{Suslin1984a, Suslin1986} implies that the natural map $K(A)\to K^\sub{top}(A)$ induces a weak equivalence $K(A)^\comp\stackrel{\sim}{\to} K^\sub{top}(A)^\comp$ (the argument can be found in the appendix of \cite{Hesselholt1997}). But profinite completion commutes with homotopy limits, and so \[K^\sub{top}(A)^\comp=\holim_r\left( K(A/\frak m^r)^\comp\right)\stackrel{(\ast)}{=}\holim_r K(A/\frak m^r),\] where the final equality follows from the previous lemma: $K(A/\frak m^r)$ has finite higher homotopy groups, hence is its own profinite completion, at least if we ignore $\pi_0$ (and thus ($\ast$) is actually only an equality if we restrict to a connected component of each side). Hence $\pi_n(K(A)^\comp)=\pi_n(K^\sub{top}(A))$ for $n>0$, whence the first claim now completes the proof.
\end{proof}

Next we consider an important example, namely rings of integers of local fields:

\begin{proposition}\label{proposition_K2_of_a_cdvr}
Let $\roi$ be a complete discrete valuation ring of mixed characteristic, with finite residue field of characteristic $p$. Let $\mu$ be the group of roots of unity inside $\roi$, and $\mu_{p^\infty}$ those of $p$-power order. Then the Hilbert symbol induces isomorphisms \[\hat K_2(\roi)\isoto\mu_{p^\infty},\quad \hat K_2(\Frac\roi)\isoto\mu.\]
\end{proposition}
\begin{proof}
Let $F=\Frac\roi$. Moore's theorem \cite{Moore1968} says that the Hilbert symbol $H:K_2(F)\to\mu$ is an isomorphism, that its kernel $\Lambda$ is an uncountable divisible group (even uniquely divisible, by \cite{Merkurjev1983}) contained inside $K_2(\roi)$, and that $K_2(\roi)/\Lambda\isoto\mu_{p^\infty}$. Set $m=\#\mu$.

According to \cite{Dennis1973a}, the kernel of $K_2(\roi)\to K_2(\roi/\frak m^r)$ is generated by Steinberg symbols of the form $\{u,1+a\}$, where $u\in\mult \roi$ and $a\in\frak m^r$. Consequently, if we pick $r$ large enough so that $1+\frak m^r\subseteq(\mult \roi)^m$, then the Hilbert symbol factors through $K_2(\roi/\frak m^r)$. Moreover, $K_2$ of a local ring is entirely symbolic (again by \cite{Dennis1973a}), so $K_2(\roi/\frak m^r)$ is finite (which we saw in the previous proposition anyway) and $K_2(\roi)\to K_2(\roi/\frak m^r)$ is surjective. This proves that the Hilbert symbol induces an isomorphism \[H:K_2(\roi/\frak m^r)\isoto\mu_{p^\infty}\] for all $r\gg0$. In conclusion, $\hat K_2(\roi)\cong\mu_{p^\infty}$.

The isomorphism for $\hat K_2(\Frac \roi)$ follows from its original definition as a pushout and the exact sequence $0\to K_2(\roi)\to K_2(F)\to \mu/\mu_{p^\infty}\to 0$.
\end{proof}

\begin{remark}
If $\roi$ is as in the previous proposition, then J.~Wagoner \cite{Wagoner1976} showed that $\hat K_n(\roi)\cong K_n(\roi/\frak m)\oplus V_n$, where $V_n$ is a \[\begin{cases}\mbox{finite $\bb Z_p$-module} & \mbox{if $n$ is even,}\\\mbox{finitely generated $\bb Z_p$-module of rank }|F : \bb Q_p| & \mbox{if $n$ is odd.}\end{cases}\]
\end{remark}

We finish this finite residue field section by showing that $\hat K_2$ is the $p$-adic completion of $K_2$:

\begin{proposition}
Let $A$ be a one-dimensional, Noetherian local ring of characteristic zero, with finite residue field of characteristic $p$. Then \[\hat K_2(A)\cong\projlim_{r\ge 1} K_2(A)/p^rK_2(A)\]
\end{proposition}
\begin{proof}
As in the proof of the previous proposition, $K_2(A)\to K_2(A/\frak m^r)$ is surjective; so $\hat K_2(A)=\projlim_r K_2(A)/E_r$, where $E_r=\ker(K_2(A)\to K_2(A/\frak m^r))$. Also as in the proof of the previous proposition, $E_r$ is generated by Steinberg symbols of the form $\{u,1+a\}$ where $u\in\mult A$ and $a\in\frak m^r$.

Given $s>0$ there exists $r\gg 0$ such that $\frak m^r\subseteq p^sA$. Hensel's lemma implies that $1+p^sA\subseteq(1+pA)^{p^s}$, so we see from the description of $E_r$ that $E_r\subseteq p^sK_2(A)$.

Conversely, for any $r>0$ we may pick $s\gg 0$ such that $(1+\frak m)^{p^s}\subseteq 1+\frak m^r$. Therefore $p^sE_1\subseteq E_r$. But $K_2$ of a finite field is trivial, so $E_1=K_2(A)$.

We have proved that the chains of subgroups $\{E_r\}_r$ and $\{p^s K_2(A)\}_s$ are commensurable, completing the proof.
\end{proof}

\subsubsection*{Residue characteristic zero}
Now we turn to $\bb Q$-algebras. Our result, based on the computations in the next section, completely describes completed $K$-groups of complete discrete valuation rings of characteristic zero; its most striking aspect is that the kernel of $\hat K_n(A)\to K_n(k)$ is entirely symbolic:

\begin{proposition}
Let $A$ be a complete discrete valuation ring with characteristic zero residue field $k$; then there is a natural split short exact sequence \[0\To\hat\Omega^{n-1}_{A,\frak m}/d\hat\Omega^{n-2}_{A,\frak m}\To\hat K_n(A)\To K_n(k)\to 0,\] where $\hat\Omega^*_{A,\frak m}:=\ker(\projlim_r\Omega^*_{A/\frak m^r}\To \Omega^*_k)$.
\end{proposition}
\begin{proof}
There is certainly a split short exact sequence \[0\to\projlim_rK_n(A/\frak m^r,\frak m/\frak m^r)\to \hat K_n(A)\to K_n(k)\to 0\tag{\dag},\] and corollary \ref{corollary_not_top_degree_of_Adams} implies that $\projlim_rK_n(A/\frak m^r,\frak m/\frak m^r)=\projlim_rK_n^{(n)}(A/\frak m^r,\frak m/\frak m^r)$. By the Goodwillie isomorphism (which respect the Adams/Hodge decompositions by \cite{Cathelineau1990}), \[K_n^{(n)}(A/\frak m^r,\frak m/\frak m^r)\cong HC_{n-1}^{(n-1)}(A/\frak m^r,\frak m/\frak m^r).\] But for any (commutative, unital) ring $R$, the standard calculation of the top degree part of cyclic homology (e.g.~\cite[Thm.~4.6.8]{Loday1992}) says that $HC_{n-1}^{(n-1)}(R)=\Omega_R^{n-1}/d\Omega_R^{n-2}$. Therefore the kernel in (\dag) is \[\projlim_r\ker(\Omega^{n-1}_{A/\frak m^r}/d\Omega^{n-2}_{A/\frak m^r}\To \Omega^{n-1}_k/d\Omega^{n-2}_k)\] and the rest of the proof simply requires chasing some projective systems.

We will use the standard notation that if $J\subseteq R$ is an ideal in a ring $R$, then $\Omega_{R,J}^m:=\ker(\Omega^m_R\to\Omega^m_{R/J})$ for $m\ge 0$. It is easy to see that if $R\to R/J$ splits, then \[\ker(\Omega^m_R/d\Omega^{m-1}_R\To\Omega^m_k/d\Omega^m_k)=\Omega^m_{R,J}/d\Omega^{m-1}_{R,J}\] So the kernel in (\dag) is $\projlim_r\Omega^{n-1}_{A/\frak m^r,\frak m/\frak m^r}/d\Omega^{n-2}_{A/\frak m^r,\frak m/\frak m^r}$. Noticing that $\Omega^m_{A/\frak m^{r+1},\frak m/\frak m^{r+1}}\to\Omega^m_{A/\frak m^r,\frak m/\frak m^r}$ is surjective for any $m\ge 0$, the projective systems \[0\to \Omega^m_{A/\frak m^r,\frak m/\frak m^r}\to\Omega^m_{A/\frak m^r}\to\Omega^m_k\to 0\tag{$r\ge 1$}\] and \[0\to d\Omega^{m-1}_{A/\frak m^r,\frak m/\frak m^r}\to \Omega^m_{A/\frak m^r,\frak m/\frak m^r}\to \Omega^m_{A/\frak m^r,\frak m/\frak m^r}/d\Omega^{m-1}_{A/\frak m^r,\frak m/\frak m^r}\to 0\tag{$r\ge 1$}\] both satisfy the Mittag-Leffler condition.  Taking the limits completes the proof.
\end{proof}

\begin{remark}
When $A$ is singular, we described the non symbolic part of the kernel of $\hat K_n(A)\to K_n(k)$ in corollary \ref{corollary_explicit decomposition}.
\end{remark}

\part*{Part II: Global theory}

\section{The Zariski cohomology of $K$-theory}\label{section_Zariski}
We now begin the second part of the paper, on global theory using $K$-theoretic ad\`eles, first in the Zariski topology. We begin by describing incomplete ad\`eles for any abelian sheaf on a one-dimensional, Noetherian scheme, before specialising to sheafifed $K$-theory and showing that these ad\`ele groups fit into a long exact Mayer-Vietoris sequence which encodes both localisation and descent for $K$-theory.

\subsection{Incomplete adeles for curves}
Let $X$ be a one-dimensional, Noetherian scheme. We will explain the natural `reparations' (or `incomplete adelic') resolution of any sheaf of abelian groups on $X$. It is simple, though not widely used. $X^i$ denotes the codimension $i$ points of $X$.

Firstly, to fix notation, if $f:W\to X$ is any morphism of schemes, then we denote by $f^*$ and $f_*$ the adjoint pair of functors on sheaves:
\[\xymatrix@=1cm{Ab(W)\ar@/^5mm/[r]^{f_*}&Ab(X)\ar@/^5mm/[l]^{f^*}}\]
Secondly, define a functor\footnote{From the point of view of higher ad\`eles, the more natural notation for this functor is $\bb{A}_X({\bf1},\cdot)$.} \[\bb{A}_1:Ab(X)\to Ab(X),\quad\cal F\mapsto\prod_{x\in X^1}i_{x*}i_x^*(\cal F),\] where $i_x:\Spec k(x)\into X$ is the natural inclusion; in other words, the sections of $\bb A_1(\cal F)$ over an open set $U\subseteq X$ are \[\bb A_1(\cal F)(U)=\prod_{x\in U^1}\cal F_x.\] There are three important properties to notice about $\bb A_1$:
\begin{enumerate}
\item It is an exact, additive functor.
\item It is coaugmented: there is always a natural morphism $\cal F\to\bb{A}_1(\cal F)$, which is an isomorphism whenever $\cal F$ vanishes on a dense open subset of $X$.
\item $\bb{A}_1(\cal F)$ is flasque.
\end{enumerate}
Now let $\cal F$ be a fixed sheaf of abelian groups on $X$. For any dense open subset $V\subseteq X$, set \[\cal F_V=j_{V*}j_V^*\cal F,\] where $j_V:V\into X$ is the natural embedding. As always there is a natural morphism $\cal F\to \cal F_V$, and the kernel and cokernel (call them $A$ and $B$ respectively) are supported on $X\setminus V$: \[0\To A\To\cal F\To\cal F_V\To B\To 0.\] Applying to this sequence the functor $\bb A_1$ and using properties (i) and (ii), we obtain a commutative diagram with exact rows:
\[\xymatrix{
0\ar[r]&A\ar[r]\ar[d]^\cong&\cal F\ar[r]\ar[d]&\cal F_V\ar[r]\ar[d]& B\ar[r]\ar[d]^\cong& 0\\
0\ar[r]&\bb{A}_1(A)\ar[r]&\bb{A}_1(\cal F)\ar[r]&\bb{A}_1(\cal F_V)\ar[r]& \bb{A}_1(B)\ar[r]& 0
}\]
Therefore the central square of this diagram is cartesian and co-cartesian in the category $Ab(X)$; it remains such after taking the limit over $V$, and so, in conclusion:
\begin{lemma}\label{lemma_bicartesian_for_Zariski}
For any $\cal F\in Ab(X)$, the following diagram is bicartesian:
\[\xymatrix{
\cal F\ar[r]\ar[d]&\indlim_V\cal F_V\ar[d]\\
\bb{A}_1(\cal F)\ar[r]&\indlim_V\bb{A}_1(\cal F_V),
}\]
where $V$ runs over all dense open subsets of $X$.

Moreover, apart from $\cal F$, the remaining three corners are flasque shaves.
\end{lemma}
\begin{proof}
All that is left to prove is that $\indlim_V\cal F_V$ is flasque; in fact, we will show that \[\indlim_V\cal F_V=\prod_{y\in X^0}i_{y*}i_y^*(\cal F).\] It is enough to check this on stalks of points in $X^1$.

Fix a dense open subset $V$, and let $x\in X\setminus V$. Then standard $(f_*, f^*)$-functoriality tells us that \[(\cal F_V)_x=f_x^*\cal F\,(D_x^\circ),\] where $D_x^\circ=\Spec\roi_{X,x}\setminus\{x\}$ is the punctured disk around $x$ and $f_x:D_x^\circ\to X$ is the natural morphism. But $D_x^\circ$ is a zero dimensional scheme with points equal to those $y\in X^0$ such that $x\in\Cl y$; thus \[(\cal F_V)_x=\prod_{\substack{y\in X^0\\\sub{s.t. }y>x}}\cal F_y=\prod_{y\in X^0}i_{y*}i_y^*(\cal F)_x,\tag{\dag}\] as required. Here we have written $y>x$ to mean that $x$ is a strict specialisation of $y$.
\end{proof}

In order to avoid too many messy expressions like (\dag), it is essential to introduce restricted product notation:

\begin{definition}[Restricted product notation]\label{definition_restricted_prod}
Let $\cal G$ be a presheaf of abelian groups on $X$. If $x\in X^1$, we will write \[\cal G_x^\circ:=f_x^*\cal G\,(D_x^\times)=\prod_{\substack{y\in X^0\\\sub{s.t. }y>x}}\cal G_y\] for the sections of $\cal G$ on the punctured spectrum $D_x^\circ$, where $f_x:D_x^\circ\to X$ is as in the previous lemma. Next we introduce typical adelic `restricted product' notation: \[\rprod_{x\in X^1}\cal G_x^\circ:=\indlim_{\substack{V\subseteq X\\\sub{dense open}}}\left(\prod_{x\in V^1}\cal G_x\times\prod_{x\in X\setminus V}G_x^\circ\right).\]

More generally, suppose that we are simply given a morphism of abelian groups $A_x\to A_x^\circ$ for each $x\in X^1$; then we may write \[\rprod_{x\in X^1}A_x^\circ:=\indlim_{\substack{V\subseteq X\\\sub{dense open}}}\left(\prod_{x\in V^1}A_x\times\prod_{x\in X\setminus V}A_x^\circ\right).\]
\end{definition}

\begin{example}
\begin{enumerate}
\item If $\cal G=\roi_X$ then $\cal G_x^\circ=\Frac\roi_{X,x}$. More generally, if $\cal G$ is a coherent $\roi_X$ module then $\cal G_x=M_x\otimes_{\roi_{X,x}}\Frac\roi_{X,x}$.
\item If $X$ is the spectrum of the ring of integers of a number field, or a smooth projective curve over a final field, then \[\rprod_{x\in X^1}\roi_{X,x}^\circ=\rprod_{x\in X^1}\Frac\roi_{X,x}=\mbox{usual incomplete ring of (finite) ad\`eles}.\]
\item The previous lemma implies that the global sections of $\indlim_V\bb A_1(\cal F_V)$ are $\rprod_{x\in X^1}\cal F_x^\circ$.
\end{enumerate}
\end{example}

In conclusion we reach the main `theorem' of one-dimensional ad\`eles:

\begin{proposition}\label{proposition_ses_for_Zariski_cohomology}
Let $\cal F$ be an abelian sheaf on $X$. Then there are natural isomorphisms \[H^*\left(0\to\prod_{x\in X^1}\cal F_x\oplus\prod_{y\in X^0}\cal F_y\to\rprod_{x\in X^1}\cal F_x^\circ\to 0\right)\cong H^*(X,\cal F).\] 
\end{proposition}
\begin{proof}
This follows by taking cohomology in lemma \ref{lemma_bicartesian_for_Zariski}.
\end{proof}

\subsection{The Zariski long exact sequence for $K$-theory}\label{subsection_Zariski_les}
$X$ continues to be a one-dimensional, Noetherian scheme, which we assume further is quasi-separated (i.e., the diagonal map $X\to X\times_{\bb Z}X$ is quasi-compact). In this section we extend some arguments from \cite{Weibel1986} to show how the localisation theorem for $K$-theory yields a long exact Mayer-Vietoris sequence on $X$, and we then compare it to the short exact sequences arising from proposition \ref{proposition_ses_for_Zariski_cohomology} with $\cal F=\cal K_n$ (sheafification of the $K_n$ presheaf in the Zariski topology).

In the Cohen-Macaulay case, our corollary \ref{corollary_descent_in_Zariski} is precisely the main theorem of \S2 of \cite{Weibel1986}. Weibel's proof used `truncations' of $K$-theoretic ad\`eles, such as in line (\ddag) in the next proof; he remarked that natural flasque resolution of $\cal K_n$ would provide an easier proof of his results, and one goal of the next proposition is to show that our adelic resolution does exactly that.

\begin{remark}
At the risk of repeating remark \ref{remark_conjecture_for_non_CM}, we comment that we have chosen to work with arbitrarily singular $X$ and therefore need the localisation theorem of Thomason-Trobaugh; if we were to restrict to Cohen-Macaualy $X$, the original Quillen-Grayson localisation theorem would suffice.

$K(X)$ means the $K$-theory spectrum of the complicial biWaldhausen category $\op{Perf}(X)$ of perfect complexes on $X$ of globally finite Tor-amplitude. See the appendix for more details.
\end{remark}

\begin{proposition}\label{proposition_MV_sequence_in_Zariski_top}
Suppose $X$ is a one-dimensional, quasi-separated, Noetherian scheme; then there is a long-exact Mayer-Vietoris sequence \[\cdots\to K_n(X)\to\prod_{x\in X^1}K_n(\roi_{X,x})\oplus\prod_{y\in X^0}K_n(\roi_{X,y})\to\rprod_{x\in X^1}K_n(\Frac\roi_{X,x})\to\cdots,\] where we use restricted product notation (definition \ref{definition_restricted_prod}) for the data $K_n(\roi_{X,x})\to K_n(\Frac\roi_{X,x})$, where $x\in X^1$.
\end{proposition}
\begin{proof}
Let $V$ be a dense open subset of $X$. Then Thomason-Trobaugh's localisation and excision theorems, and passing to the limit, yield a fibre sequence of $K$-theory spectra \[K(X\mbox{ on }X\setminus V)\to K(X)\to K(V)\tag{\dag}\] and a homotopy equivalence \[K(X\mbox{ on }X\setminus V)\xto{\sim}\prod_{x\in X\setminus V} K(\Spec\roi_{X,x}\mbox{ on }x)\] For each $x\in X\setminus V$ we also have the local fibre sequence \[K(\Spec\roi_{X,x}\mbox{ on }x)\to K(\roi_{X,x})\to K(\Frac\roi_{X,x}).\] Taking the product over all $x\in X\setminus V$ and comparing with fibre sequence (\dag) reveals that 
\[\xymatrix{
K(X)\ar[r]\ar[d]& K(V)\ar[d]\\
\prod_{x\in X\setminus V}K(\roi_{X,x})\ar[r]&\prod_{x\in X\setminus V}K(\Frac\roi_{X,x})
}\] is a homotopy cartesian square of spectra, so that it yields a long-exact Mayer-Vietoris sequence of the homotopy groups: \[\cdots\to K_n(X)\to \prod_{x\in X\setminus V}K_n(\roi_{X,x})\oplus K_n(V)\to\prod_{x\in X\setminus V}K_n(\Frac\roi_{X,x})\to\cdots\tag{\ddag}\]  Evidently this remains exact if we artificially add a factor of $\prod_{x\in V^1}K_n(\roi_{X,x})$ to consecutive terms: \[\cdots\to K_n(X)\to \prod_{x\in X^1}K_n(\roi_{X,x})\oplus K_n(V)\to\prod_{x\in V^1}K_n(\roi_{X,x})\times\prod_{x\in X\setminus V}K_n(\Frac\roi_{X,x})\to\cdots\] 

The proof is completed by taking the limit over dense opens $V$.
\end{proof}

From the proposition we obtain the following local formulae for the cohomology groups $H^i(X,\cal K_n)$:

\begin{corollary}\label{corollary_ker_coker_description_of_K_groups}
Let $X$ be as in the proposition; then $H^0(X_\sub{Zar},\cal K_n)$ and $H^1(X_\sub{Zar},\cal K_{n+1})$ are equal to the image and kernel, respectively, of the diagonal map \[K_n(X)\to \prod_{x\in X^1}K_n(\roi_{X,x})\oplus \prod_{y\in X^0}K_n(\roi_{X,y}).\]
\end{corollary}
\begin{proof}
Proposition \ref{proposition_ses_for_Zariski_cohomology} for the sheaf $\cal F=\cal K_n$ gives us an exact sequence
\[0\to H^0(X_\sub{Zar},\cal K_n)\to\prod_{x\in X^1} K_n(\roi_{X,x})\oplus\prod_{y\in X^0}K_n(\roi_{X,y})\to\rprod_{x\in X^1} K_n(\Frac\roi_{X,x})\to H^1(X_\sub{Zar},\cal K_n)\to 0.\]
This gives us the formula for $H^0$. For $H^1$, combine this exact sequence with the long exact sequence of the previous proposition.
\end{proof}

From the previous corollary we obtain the descent spectral sequence on $X$ for $K$-theory:

\begin{corollary}\label{corollary_descent_in_Zariski}
There are natural exact sequences \[0\to H^1(X_\sub{Zar},\cal K_{n+1})\to K_n(X)\to H^0(X_\sub{Zar},\cal K_n)\to 0.\]
\end{corollary}
\begin{proof}
This is immediate from the previous corollary.
\end{proof}

\subsection{Relation to descent}\label{subsection_Zariski_descent}
$X$ continues to be a one-dimensional, quasi-separated, Noetherian scheme. Let $E$ be a presheaf of spectra on $X_\sub{Zar}$, with associated abelian presheaves $E_n=\pi_n(E)$; let $\cal E_n$ be the sheafification of $E_n$. We say that $E$ satisfies descent when it transforms cartesian squares \xysquare{U\cup V}{U}{V}{U\cap V}{->}{->}{->}{->} of open subsets of $X$ to homotopy cartesian squares of spectra, which is (more or less) equivalent to the existence of a descent spectral sequence \[E_2^{p,q}=H^p(X,\cal E_q)\Longrightarrow E_{p-q}(X).\] Since the Zariski site of $X$ has cohomological dimension one, this is the assertion that there is a natural exact sequence \[0\to H^1(X,\cal E_{n+1})\to E_n(X)\to H^0(X,\cal E_n)\to 0.\]

This yields a reinterpretation of the calculations of the previous subsection:

\begin{proposition}
Let $X$ be a one-dimensional, Noetherian scheme, and let $E$ be a presheaf of spectra on $X$. Then the following are equivalent:
\begin{enumerate}
\item $E$ satisfies descent (more precisely, there is a short exact sequence as immediately above);
\item There exists a long exact Mayer-Vietoris sequence \[\cdots\to E_n(X)\to\prod_{x\in X^1}E_{n,x}\oplus\prod_{y\in X^0}E_{n,y}\to\rprod_{x\in X^1}E_{n,x}^\circ\to\cdots\] as in proposition \ref{proposition_MV_sequence_in_Zariski_top}.
\item For $n\ge 0$, the cohomology groups $H^0(X_\sub{Zar},\cal E_n)$ and $H^1(X_\sub{Zar},\cal E_{n+1})$ are equal to the image and kernel, respectively, of the diagonal map \[E_n(X)\to \prod_{x\in X^1}E_{n,x}\oplus \prod_{y\in X^0}E_{n,y}.\]
\end{enumerate}
\end{proposition}
\begin{proof}
Proposition \ref{proposition_ses_for_Zariski_cohomology} tells us that there is an exact sequence \[0\to H^0(X,\cal E_n)\to\prod_{x\in X^1}E_{n,x}\oplus\prod_{y\in X^0}E_{n,y}\to\rprod_{x\in X^1} E_{n,x}^\circ\to H^1(X,\cal E_n)\to 0\] for any $n\ge 0$. The equivalence of conditions (i)--(iii) follows in an elementary way from this by splicing and unravelling exact sequences; we only sketch the details, abbreviating notation a little to save space:

(i)$\Rightarrow$(ii) Assuming (i), we have a commutative diagram with exact rows:
\[\hspace{-14mm}\xymatrix@=3mm{
\ar@{..>}[r]&H^1(\cal E_{n+1})\ar[r]\ar[d]&0\ar[r]\ar[d] &H^0(\cal E_n)\ar[r]\ar[d]& \prod_xE_{n,x}\oplus\prod_yE_{n,y}\ar[r]\ar[d]&\rprod_{x} E_{n,x}^\circ\ar[r]\ar[d]& H^1(\cal E_n)\ar[r]\ar[d]& 0\ar[r]\ar[d]& H^0(\cal E_{n-1})\ar@{..>}[r]\ar[d]&\\
\ar@{..>}[r]&H^1(\cal E_{n+1})\ar[r]&E_n(X)\ar[r]&H^0(\cal E_n)\ar[r]&0\ar[r]&0\ar[r]&H^1(\cal E_n)\ar[r]&E_{n-1}(X)\ar[r]&H^0(\cal E_{n-1})\ar@{..>}[r]&
}\]
The exact sequence (ii) follows by a standard Mayer-Vietoris type diagram chase.

(ii)$\Rightarrow$(iii) follows exactly as in corollary \ref{corollary_ker_coker_description_of_K_groups}. (iii)$\Rightarrow$(i) is trivial.
\end{proof}

\section{The Nisnevich cohomology of $K$-theory}\label{section_Nisnevich_case}
Now we repeat all the constructions above in the Nisnevich topology; apart from several steps, the arguments are exactly the same and so we do not linger. Then finally we reach section \ref{subsection_Nisnevich_cohomology_via_adeles}, where it is shown that the arguments can be modified again to give a description of the Nisnevich cohomology of sheafified $K$-theory using our completed $K$-groups from the first part of the paper. The validity of this condition is dependent on conjecture 1 being satisfied.

$X$ is again a one-dimensional, Noetherian scheme.

\subsection{Henselian adeles for curves}\label{subsection_Henselian_adeles}
Here we describe the Henselian adeles, which provide a functorial resolution for any abelian sheaf $\cal F$ on $X_\sub{Nis}$.

As before, given a morphism $f:W\to X$, we denote by $f^*$ and $f_*$ the adjoint pair of functors on abelian sheaves, this time in the Nisnevich topology:
\[\xymatrix@=1cm{Ab(W_\sub{Nis})\ar@/^5mm/[r]^{f_*}&Ab(X_\sub{Nis})\ar@/^5mm/[l]^{f^*}}.\] Let the functor $\bb{A}_1^\sub{Nis}:Ab(X_\sub{Nis})\to Ab(X_\sub{Nis})$ be defined as before: $\bb{A}_1^\sub{Nis}(\cal F)=\prod_{x\in X^1}i_{x*}i_x^*(\cal F)$, where $i_x:\Spec k(x)\into X$ is the natural map.

Exactly imitating the arguments in the Zariski case shows that we have a bicartesian square of Nisnevich sheaves
\[\xymatrix{
\cal F\ar[r]\ar[d]&\prod_{y\in X^0}i_{y*}i_y^*(\cal F)\ar[d]\\
\bb{A}_1^\sub{Nis}(\cal F)\ar[r]&\indlim_V\bb{A}_1^\sub{Nis}(\cal F_V)
}\]
where $V$ runs over all \'etale opens of $X$ such that $X^0\subseteq\op{cd}(V/X)$ ($=\{x\in X:\exists\, x'\in V\mbox{ s.t. }x'\mapsto x\mbox{ and }k(x)\isoto k(x')\}$), and where $\cal F_V$ is defined as in the Zariski case: $\cal F_V=j_{V*}j_V^*(\cal F)$, with $j_V:V\to X$ being the natural inclusion.

Moreover, note that any \'etale cover $V\to X$ such that $X^0\subseteq\op{cd}(V/X)$ may be refined to a Zariski cover $U\to X$ with the same property; i.e., such that $U$ is dense. It then follows from the same argument as in the Zariski case that \[\bb A_1^\sub{Nis}(\cal F_U)=\prod_{x\in U^1}\cal F_x\times\prod_{x\in X\setminus U}\cal F_x^\circ,\] where $\cal F_x^\circ$ is now the global sections of $\cal F_x$ on the punctured Henselian spectrum $\Spec\roi_{X,x}^h\setminus\{x\}$. Therefore \[\indlim_V\bb{A}_1^\sub{Nis}(\cal F_V)=\rprod_{x\in X^1}\cal F_x^\circ.\]

\begin{example}
If $X$ is the spectrum of the ring of integers of a number field, or a smooth projective curve over a field, and $\cal G=\roi_X$ (as a Nisnevich sheaf), then $\cal G_x^\circ=\Frac\roi_{X,x}^h$ and $\rprod_{x\in X^1}\Frac\roi_{X,x}$ is the usual ring of (finite) Henselian ad\`eles.
\end{example}

Taking cohomology we obtain the analogue of proposition \ref{proposition_ses_for_Zariski_cohomology}:

\begin{proposition}\label{proposition_ses_for_Nis}
Let $\cal F$ be an abelian sheaf on $X_\sub{Nis}$. Then there are natural isomorphisms \[H^*\left(0\to\prod_{x\in X^1}\cal F_x\oplus\prod_{y\in X^0}\cal F_y\to\rprod_{x\in X^1}\cal F_x^\circ\to0\right)\cong H^*(X_\sub{Nis},\cal F).\]
\end{proposition}

\subsection{The Nisnevich long exact sequence for $K$-theory}\label{subsection_Nisnevich_les}
Assume further that $X$ is quasi-separated. We showed in proposition \ref{proposition_MV_sequence_in_Zariski_top} that there is a natural long exact Mayer-Vietoris sequence: \[\cdots\to K_n(X)\to\prod_{x\in X^1}K_n(\roi_{X,x})\oplus\prod_{y\in X^0}K_n(\roi_{X,y})\to\rprod_{x\in X^1}K_n(\roi_{X,y})\to\cdots\] We now wish to replace the local rings $\roi_{X,x}$ in this diagram by their Henselizations.

Given $x\in X^1$, the homomorphism $\roi_{X,x}\to\roi_{X,x}^h$ is an isomorphism infinitely near the maximal ideals (in the sense of Thomason-Trobaugh), so it gives rise to a long exact Mayer-Vietoris sequence: \[\cdots\to K_n(\roi_{X,x})\to K_n(\roi_{X,x}^h)\oplus K_n(\Frac\roi_{X,x})\to K_n(\Frac\roi_{X,x}^h)\to\cdots\] This may be spliced with the long exact sequence immediately above (or the arguments in the previous section may be repeated verbatim) to give a long exact sequence \[\cdots\to K_n(X)\to\prod_{x\in X^1}K_n(\roi_{X,x}^h)\oplus\prod_{y\in X^0}K_n(\roi_{X,y})\to\rprod_{x\in X^1}K_n(\Frac\roi_{X,x}^h)\to\cdots\]

We may now easily deduce the Nisnevich analogue of corollary \ref{corollary_ker_coker_description_of_K_groups}:

\begin{corollary}\label{corollary_ker_coker_description_of_K_groups_Nis}
$X$ a one-dimensional, quasi-separated, Noetherian scheme. Then $H^0(X_\sub{Nis},\cal K_n)$ and $H^1(X_\sub{Nis},\cal K_{n+1})$ are equal to the image and kernel, respectively, of the diagonal map \[K_n(X)\to \prod_{x\in X^1}K_n(\roi_{X,x}^h)\oplus \prod_{y\in X^0}K_n(\roi_{X,y}).\]
\end{corollary}
\begin{proof}
This follows exactly as it did in the Zariski setting since proposition \ref{proposition_ses_for_Nis} provides us with exact sequences
\[\hspace{-10mm}0\to H^0(X_\sub{Nis},\cal K_n)\to\prod_{x\in X^1} K_n(\roi_{X,x}^h)\oplus\prod_{y\in X^0}K_n(\roi_{X,y})\to\rprod_{x\in X^1} K_n(\Frac\roi_{X,x}^h)\to H^1(X_\sub{Nis},\cal K_n)\to 0\] and since we just established the Nisnevich analogue of proposition  \ref{proposition_MV_sequence_in_Zariski_top}.
\end{proof}

\subsection{Relation to descent}
Section \ref{subsection_Zariski_descent} has an obvious Nisnevich modification.

\subsection{Nisnevich cohomology via $K$-theory of completions}\label{subsection_Nisnevich_via_completions}
$X$ continues to be a one-dimensional, quasi-separated, Noetherian scheme. Here we explain that one can always replace $K_n(\roi_{X,x}^h)$ by $K_n(\hat \roi_{X,x})$ in expressions for the Nisnevich cohomology of $K$-theory.

\begin{lemma}\label{lemma_henselian_to_complete}
The diagram
\xysquare{\prod_{x\in X^1}K_n(\roi_{X,x}^h)\oplus\prod_{y\in X^0}K_n(\roi_{X,y})}{\rprod_{x\in X^1}K_n(\Frac\roi_{X,x}^h)}{\prod_{x\in X^1}K_n(\hat \roi_{X,x})\oplus\prod_{y\in X^0}K_n(\roi_{X,y})}{\rprod_{x\in X^1}K_n(\Frac\hat\roi_{X,x})}{->}{^(->}{^(->}{->}
is bicartesian with injective vertical arrows.
\end{lemma}
\begin{proof}
Corollary \ref{corollary_equivalence_of_completed_and_Henselian} shows that if $x\in X^1$ then the diagram \xysquare{K_n(\roi_{X,x}^h)}{K_n(\Frac\roi_{X,x}^h)}{K_n(\hat \roi_{X,x})}{K_n(\Frac\hat\roi_{X,x})}{->}{^(->}{^(->}{->} is bicartesian with injective vertical arrows. For $V\subseteq X$ a dense open, we take $\prod_{x\in X\setminus V}$ of these diagrams to obtain a new bicartesian diagram; then apply $\times\prod_{x\in V^1}K_n(\roi_{X,x}^h)$ to the top row of the diagram and $\times\prod_{x\in V^1}K_n(\hat\roi_{X,x})$ to the bottom of the diagram. Finally take $\indlim_V$.
\end{proof}

It follows at once from the lemma that we may replace each $\roi_{X,x}^h$ by $\hat\roi_{X,x}$ in both the long exact Mayer-Vietoris sequence appearing immediately before the proof of corollary \ref{corollary_ker_coker_description_of_K_groups_Nis} and in the corollary itself:
\[\cdots\to K_n(X)\to\prod_{x\in X^1}K_n(\hat\roi_{X,x})\oplus\prod_{y\in X^0}K_n(\roi_{X,y})\to\rprod_{x\in X^1}K_n(\Frac\hat\roi_{X,x})\to\cdots\]
and
\[H^*\left(0\to\prod_{x\in X^1} K_n(\hat\roi_{X,x})\oplus\prod_{y\in X^0}K_n(\roi_{X,y})\to\rprod_{x\in X^1} K_n(\Frac\hat\roi_{X,x})\to 0\right) \cong H^*(X_\sub{Nis},\cal K_n)\]
Therefore $H^0(X_\sub{Nis},\cal K_n)$ and $H^1(X_\sub{Nis},\cal K_{n+1})$ are equal to the image and kernel, respectively, of the diagonal map \[K_n(X)\to \prod_{x\in X^1}K_n(\hat\roi_{X,x})\oplus \prod_{y\in X^0}K_n(\roi_{X,y}).\]

\subsection{Nisnevich cohomology via completed $K$-groups}\label{subsection_Nisnevich_cohomology_via_adeles}
$X$ continues to be a one-dimensional, quasi-separated, Noetherian scheme, but we now assume further that the completion of $\roi_{X,x}$ satisfies conjecture 1 for all $x\in X^1$ (or, equivalently, that its Henselization satisfies conjecture 1'). For example, theorem \ref{theorem_char_zero_local_result} implies that $X$ is allowed to be a reduced curve over a characteristic zero field.

Then the square \xysquare{K_n(\roi_{X,x}^h)}{K_n(\Frac\roi_{X,x}^h)}{\hat K_n(\roi_{X,x})}{\hat K_n(\Frac\roi_{X,x})}{->}{->}{->}{->} is bicartesian for each $x\in X^1$. The arguments of section \ref{subsection_Nisnevich_via_completions} may then be repeated verbatim (the injectivity of the vertical arrows in lemma \ref{lemma_henselian_to_complete} was of no importance) to obtain our main theorem on calculating the cohomology of $K$-theory adelically:

\begin{theorem}\label{theorem_main}
Let $X$ be a one-dimensional, quasi-separated, Noetherian scheme such that $\hat\roi_{X,x}$ satisfies conjecture $1$ for all $x\in X^1$. Then there is a long exact Mayer-Vietoris sequence \[\cdots\to K_n(X)\to\prod_{x\in X^1}\hat K_n(X)\oplus \prod_{y\in X^0} K_n(\roi_{X,y})\to\rprod_{x\in X^1}\hat K_n(\Frac\roi_{X,x})\to\cdots\]
and natural isomorphisms\[H^*\left(0\to \prod_{x\in X^1} \hat K_n(\roi_{X,x})\oplus\prod_{y\in X^0}K_n(\roi_{X,y})\to\rprod_{x\in X^1} \hat K_n(\Frac\roi_{X,x})\to 0\right)\cong H^*(X_\sub{Nis},\cal K_n).\] Therefore $H^0(X_\sub{Nis},\cal K_n)$ and $H^1(X_\sub{Nis},\cal K_{n+1})$ are equal to the image and kernel, respectively, of the diagonal map \[K_n(X)\to \prod_{x\in X^1}\hat K_n(\roi_{X,x})\oplus \prod_{y\in X^0}K_n(\roi_{X,y}).\]
\end{theorem}


\begin{appendix}

\section{K-theory and Hochschild/cyclic homology}
This appendix summarises the various tools from $K$-theory and Hochschild/cyclic homology which are employed in the paper.

\subsection{$K$-theory}
$K$-theory in this paper is in the style of T.~Thomason and R.~Trobaugh \cite{Thomason1990}. However, our manipulations are mostly formulaic and the precise definitions do not matter a great deal; indeed, apart from the possibility of non-Cohen-Macaulay one-dimensional local rings, `classical' $K$-theory would suffice. Therefore this summary is only for the sake of completeness.

If $X$ is a scheme then $K(X)$ denotes the $K$-theory spectrum of the complicial biWaldhausen category $\op{Perf}(X)$ of perfect complexes on $X$ of globally finite Tor-amplitude, and $K_n(X):=\pi_n(K(X))$. This agrees with the `naive' definition using the category of locally-free coherent $\roi_X$-modules as soon as $X$ has an ample family of line bundles (`divisorial' in the language of SGA 6 II 2.2), e.g.~quasi-projective over an affine scheme, or regular+separated+Noetherian. We will not worry about exactness at the $K_0$ end of all our long exact Mayer-Vietoris sequences, and so do not replace $K$-theory by non-connective Bass $K^B$-theory.

If $Y$ is a closed subscheme of $X$, then $K(X\mbox{ on }Y)$ is the $K$-theory spectrum of the subcategory $\op{Perf}_Y(X)$ of $\op{Perf}(X)$ consisting of those complexes which are acyclic on $X\setminus Y$. Thomason and Trobaugh introduce the notion of when a morphism of schemes $f:X'\to X'$ is an `isomorphism infinitely near' $Y\subseteq X$; when this is satisfied, the resulting map $f^*:K(X\mbox{ on }Y)\to K(X'\mbox{ on }f^{-1}(Y))$ is an equivalence. Most importantly for us, if $A$ is a Noetherian ring and $I\subseteq A$, then \[A\to\mbox{the completion or Henselization of $A$ at $I$}\] are isomorphisms infinitely near $I$ \cite[3.19.2]{Thomason1990}.

\subsection{Karoubi-Villamayor $K$-theory}
Given a ring $R$, let $R[\Delta^\bullet]$ be the usual simplicial ring, which in degree $n$ is equal to \[R[T_0,\dots,T_n]/\pid{\sum_iT_i=1}.\] Then $GL(R[\Delta^\bullet])$ is a simplicial group and the Karoubi-Villamayor $K$-theory of $R$ is defined by \[KV_n(R)=\pi_n(BGL(R[\Delta^\bullet]).\tag{$n\ge1$}\] There are natural maps $K_n(R)\to KV_n(R)$, for $n\ge 1$, arising as edge maps in a first quadrant spectral sequence \[E^1_{pq}=K_q(R[\Delta^p])\Rightarrow KV_{p+q}(R)\tag{$p\ge 0,\,q\ge1$},\] called the Andersen spectral sequence.

In particular, if $R$ is $K_1$-regular, i.e.~$K_1(R)\isoto K_1(R[T_1])\isoto K_1(R[T_1,T_2]\isoto\cdots$, then we obtain a short exact sequence \[NK_2(R)\to K_2(R)\to KV_2(R)\to 0,\] where $NK_2(R)=\op{coker}(K_2(R)\to K_2(R[T]))$.

$KV$-theory has much better excision type properties than usual $K$-theory. A {\em $GL$-fibration} is a homomorphism of rings $f:R\to R'$ such that \[GL(R[t_1,\dots,t_n])\times GL(R')\to GL(R'[t_1,\dots,t_n])\] is surjective for all $n\ge 1$. If $f$ is a $GL$-fibration then it is surjective, and the converse is true if $R'$ is regular. Suppose that $i:S\to R$ is a homomorphism of rings and $I\subseteq S$ is an ideal of $R$ mapped isomorphically onto an ideal of $R$; if $R\to R/i(I)$ is a $GL$-fibration then there is a long-exact Mayer-Vietories sequence \[\cdots\to KV_n(S)\to KV_n(S/I)\oplus KV_n(R)\to KV_n(R/i(I))\to\cdots\]

\subsection{Hochschild/cyclic homology}
Let $k$ be a commutative ring and $R$ a $k$-algebra (again, for us this will always be commutative and unital).  Define a simplicial $R$-algebra $C_\bullet(R)$ by
\begin{align*}
C_n(R)&=\underbrace{R\otimes_k\cdots\otimes_kR}_{\sub{$n+1$ copies}}\\
d_i(a_0\otimes\cdots\otimes a_n)&=\begin{cases} a_0\otimes\cdots\otimes a_ia_{i+1}\otimes\cdots\otimes a_n&(i=0,\dots,n-1)\\ a_na_0\otimes\cdots\otimes a_{n-1}&(i=n)\end{cases}\\
s_i(a_0\otimes\cdots\otimes a_n)&=a_0\otimes\cdots\otimes a_i\otimes 1\otimes a_{i+1}\otimes \cdots\otimes a_n\quad (i=0,\dots,n)
\end{align*}
From now on we write $C_\bullet(R)$ for the usual complex of $R$-modules associated to the simplicial $R$-algebra $C_\bullet(R)$. The Hochschild homology of $R$ is the homology of this complex: $HH_n(R):=H_n(C_\bullet(R))$.

The cyclic group $\bb Z/(n+1)\bb Z$ naturally acts on $C_n(R)$ by permutation, and we let \[C_n^\lambda(R)=C_n(R)/(1-t_n)C_n(R)\tag{$t_n$ a generator of $\bb Z/(n+1)\bb Z$}\] denote the coinvariants of the action. The differential on $C_\bullet(R)$ descends to give {\em Connes' complex} $C_\bullet^\lambda(R)$. Assuming that $\bb Q\subseteq R$ (which will always be the case for us), the cyclic homology of $R$ may be defined to be the homology of Connes' complex: $HC_n(R):=H_n(C_\bullet^\lambda(R))$.

If $R=R_0\oplus R_1\oplus\cdots$ is a graded $k$-algebra, then we write $\tilde{HH}_n(R)$ and $\tilde{HC}_n(R)$ for the reduced homology groups $HH_n(R)/HH_n(R_0)$ and $HC_n(R)/HC_n(R_0)$.

\section{The $K$-theory of seminormal local rings}
The purpose of this appendix is to collect various classical results on seminormal local rings and their $K$-theory \cite{Davis1978, Roberts1976, Weibel1980, Weibel1989}.

\subsection{Seminormal local rings}
Let $A$ be a one-dimensional, reduced ring whose normalisation $\tilde A$ is a finitely generated $A$-module (this is automatic is $A$ is excellent, which we will tacitly assume from now on). Then $A$ is said to be {\em seminormal} when the following equivalent conditions hold:
\begin{enumerate}
\item $\op{Pic}(A)\to\op{Pic}(A[T])$ is an isomorphism.
\item The conductor ideal $\frak c\subseteq A$ is a radical ideal in $\tilde A$; i.e.~$\tilde A/\frak c$ is a reduced ring (hence a finite product of fields since it is Artinian).
\item If $f\in A$ satisfies $f^2,f^3\in A$, then $f\in A$.
\end{enumerate}
Put \[A^+=\{f\in \tilde A:f^2\mbox{ and }f^3\mbox{ are in }A\}.\] Then $A^+$ is the smallest seminormal subring of $\tilde A$ containing $A$, and it is called the {\em seminormalisation} of $A$. Each maximal ideal of $A$ sits under a unique maximal ideal of $A^+$; i.e.~$\Spec A\to\Spec A^+$ is a bijection.

Now suppose $A$ (still one-dimensional, reduced, and excellent) is moreover local with residue field $k$. Let $\frak m_1,\dots\frak m_n$ be the distinct maximal ideals of $\tilde A$, and let $\frak M=\bigcap_{i=1}^n\frak m_1$ be its Jacobson radical. Then clearly $A$ is seminormal if and only if $\frak M$ coincides with the maximal ideal of $A$. In the local case, the seminormalisation can be explicty described:

\begin{lemma}
Let $A$ be a one-dimensional, excellent, reduced local ring with residue field $k$. Let $A^0$ be the set of those $f\in\tilde A$ satisfying the following:
\begin{enumerate}
\item For every maximal ideal $\frak m$ of $\tilde A$, $f$ mod $\frak m$ lies in $k$.
\item The value of $f$ mod $\frak m$ does not depend on which maximal ideal $\frak m$ was chosen.
\end{enumerate}
Then $A^0$ is the seminormalisation of $A$.
\end{lemma}
\begin{proof}
First notice that $A\subseteq A^0$. Next let $f\in\tilde A$; it is easy to see that if $f^2$ and $f^3$ both satisfy (i) and (ii), then so does $f$. Therefore $A^0$ is seminormal; i.e.~$A^+\subseteq A^0$.

Finally, notice that $A^0$ is local with residue field $k$: its maximal ideal is $\frak M$, the Jacobson radical of $\tilde A$. Let $f\in A^0$; we will show that $f\in A^+$. We mentioned above that the maximal ideal of $A^+$ equals $\frak M$, so if $f\in\frak M$ then there is nothing more to show. Else $f$ is a unit in $A^0$; as $A^0/\frak M=k$, there is $\al\in\mult A$ such that $\al f\in1+\frak M$. Again, since $\frak M\subset A^+$, it follows that $f\in A^+$; i.e.~$A^0\subseteq A^+$.
\end{proof}

\begin{remark}
Let $A$ be as in the lemma, and suppose it is seminormal; suppose further that $\tilde A$ happens also to be local. Then the maximal ideal of $\tilde A$ equals the maximal ideal of $A$ and so the following are equivalent:
\begin{enumerate}
\item $k(A)=k(\tilde A)$;
\item $A$ is normal.
\end{enumerate}
An example to have in mind is $A=k+tK[[t]]$, where $K/k$ is a finite extension of fields. Then $A$ is seminormal with integral closure $\tilde A=K[[t]]$; the residue fields are $k(A)=k$ and $k(\tilde A)=K$.
\end{remark}

Suppose that $A$ is as in the lemma. Regardless of seminormality, the minimal prime ideals $\{\frak q\}$ of $\hat A$ are in bijective correspondence with the maximal ideals $\{\frak m\}$ of $\tilde A$ via \[\frak m_i\longleftrightarrow\frak q=\ker\pid{\hat A\to\hat{(\tilde A)_{\frak m}}}.\] See, e.g., \cite[Thm.~6.5]{Dieudonne1967}. For clarity, write $C=\hat A$, a one-dimensional, reduced, complete local ring, which is seminormal if and only if $A$ is. Let $\frak q_1,\dots,\frak q_n$ be the minimal prime ideals of $C$ and $\frak m_1,\dots,\frak m_n$ be the corresponding maximal ideals of $\tilde C=\prod_i\tilde C_{\frak m_i}$ (each $\tilde C_{\frak m_i}$ is a complete discrete valuation ring and is the normalisation of $C/\frak q_i$). Then we have a diagram of inclusions
\[\xymatrix@M=2mm{
C\ar@{^(->}[r]\ar@{^(->}[d]&\tilde C\ar@{=}[d]\\
\prod_iC/\frak q_i\ar@{^(->}[r]&\prod_i\tilde C_{\frak m_i}
}\]
and, by the previous lemma, \[C^+=\{(f_1,\dots,f_n)\in\tilde C:f_1\mbox{ mod }\frak m_1=\cdots=f_n\mbox{ mod }\frak m_n\in k\}.\]

\begin{lemma}\label{lemma_structure_of_complete_seminormal_rings}
Maintain notation from the previous paragraph, and suppose $C$ is seminormal. For $i=1,\dots,n$, set $I_i=\bigcap_{j\neq i}\frak q_j$. Then $I_1+\dots+I_n=\frak m_C$ and the sum is direct.

In particular, if $C$ contains a coefficient field $k$, then $C\cong k\oplus I_1\oplus\cdots\oplus I_m$ and $k+I_i\cong C/\frak q_i$.
\end{lemma}
\begin{proof}
The sum is direct because, for any $i$, \[I_i\cap\sum_{j\neq i}I_j\subseteq\bigcap_j\frak q_j=\{0\}.\] By the diagram just before the lemma, any element in the maximal ideal of $C=C^+$ can be written $f_1+\dots+f_n$ where each $f_i$ is zero in $C/\frak q_j$ for $j\neq i$; i.e.~$f_i\in I_i$, as required.

The claims in the equal-characteristic case follow at once.
\end{proof}

With $C$ as in the lemma, it is easy to see that each $C/\frak q_i$ is also seminormal. Since $\tilde{C/\frak q_i}=\tilde C_{\frak m_i}$, we may apply the previous remark to deduce that $C/\frak q_i$ is a complete discrete valuation ring if and only if $\tilde C/\frak m_i=k$.

\subsection{Calculations in $K$-theory related to seminormal rings}
Suppose that $A\subseteq B$ is an inclusion of rings, and that $I\subseteq A$ is an ideal of $B$ contained inside $A$, so that 
\[\xymatrix{
A\ar[r]\ar[d] & B\ar[d]\\
A/I\ar[r] & B/I
}\]
is Cartesian, as a square of abelian groups. Suppose further that $k:=A/I$, $K:=B/I$, and $B$ are regular rings.

It is essential for what is to follow to observe that these assumptions remain valid if we replace $A$ by $A[X_1,\dots,X_n]$; $B$ by $B[X_1,\dots,X_n]$; $k$ by $k[X_1,\dots,X_n]$; and $K$ by $K[X_1,\dots,X_n]$.

Since $B\to K$ is a surjection to a regular ring, it is a $GL$-filbration and therefore the ideal $I$ satisfies excision for $KV$-theory with respect to $A$ and $B$; i.e. $KV_*(A,I)\isoto KV_*(B,I)$. But also, $B$ and $K$ are regular, so $K_*(B,I)\isoto KV_*(B,I)$. In conclusion, \[KV_*(A,I)\isoto K_*(B,I),\] and we have a long exact sequence \[\cdots\to K_*(B,I)\to KV_*(A)\to K_*(k)\to\cdots\]

\begin{lemma}[\cite{Weibel1980} \cite{Vorst1979}]
Assume that $K_i(A)\to K_i(k)$ is surjective for $i=1,2$. Then $A$ is $K_1$-regular if and only if $\Omega_{K/k}^1=0$.
\end{lemma}
\begin{proof}
First notice that $A$ is $K_1$ regular if and only if $K_1(A)\to KV_1(A)$ is an isomorphism: `if' follows from the homotopy invariance of Karoubi-Villimayor $K$-theory, while `only if' follows from examining the Anderson spectral sequence.

Secondly, note that our surjectivity assumption clearly implies that the map $KV_i(A)\to KV_i(k)=K_i(k)$ is also surjective for $i=1,2$. The long exact sequence for relative $K$-theory and the long exact sequence just before the lemma therefore break into short exact sequences:
\[\xymatrix{
0\ar[r]&K_1(A,I)\ar[d]\ar[r]&K_1(A)\ar[d]\ar[r]&K_1(k)\ar[d]\ar[r]&0\\
0\ar[r]&K_1(B,I)\ar[r]&KV_1(A)\ar[r]&K_1(k)\ar[r]&0
}\tag{\dag}\]
So, $K_1(A)\to KV_1(A)$ is an isomorphism if and only if $K_1(A,I)\to K_1(B,I)$ is an isomorphism.

According to Weibel-Geller \cite{Geller1980}, there is a natural exact sequence \[K_2(B,I)\to \Omega_{B/A}^1\otimes_B I/I^2\to K_1(A,I)\to K_1(B,I)\to0,\] so if $\Omega_{B/A}^1\otimes_BI/I^2=0$ then we have shown $K_1(A,I)\to K_1(B,I)$ is an isomorphism, i.e.~$A$ is $K_1$-regular. We will now prove the converse of this statement.

Assume $A$ is $K_1$-regular. Then $K_1(A,I)\to K_1(A[\ul X],I[\ul X])$ is an isomorphism for any number of variables $\ul X=X_1,\dots,X_m$. So, comparing the Weibel-Geller sequences for $A,B$ and $A[\ul X]$, $B[\ul X]$, we have
\[\hspace{-10mm}\xymatrix{
K_2(B,I)\ar[r]\ar[d]^{\cong}& \Omega_{B/A}^1\otimes_B I/I^2\ar[r]\ar[d]& K_1(A,I)\ar[r]\ar[d]^{\cong}& K_1(B,I)\ar[d]^{\cong}\\
K_2(B[\ul X],I[\ul X])\ar[r]& \Omega_{B[\ul X]/A[\ul X]}^1\otimes_{B[\ul X]} I[\ul X]/I[\ul X]^2\ar[r]& K_1(A[\ul X],I[\ul X])\ar[r]& K_1(B[\ul X],I[\ul X])
}\]
where the left and right-most vertical arrows are isomorphisms due to regularity assumptions. A diagram chase reveals that the remaining vertical arrow, between the differential forms, is surjective; but this arrow can be rewritten as \[\Omega_{B/A}^1\otimes_BI/I^2\to(\Omega_{B/A}^1\otimes_B I/I^2)\otimes_KK[\ul X].\] Clearly this is surjective if and only if $\Omega_{B/A}^1\otimes_B I/I^2=0$.

This completes the proof of our claim that $A$ is $K_1$-regular if and only if $\Omega_{B/A}^1\otimes_B I/I^2=0$. To finish, observe that because $B$ and $B/I$ are regular, the morphism $\Spec B/I\to\Spec B$ is automatically a regular intersection and thus $I/I^2$ is a locally free $K$-module. So the vanishing of $\Omega_{B/A}^1\otimes_B I/I^2=\Omega_{K/k}^1\otimes_KI/I^2$ is equivalent to that of $\Omega_{K/k}^1$.
\end{proof}

If $A$ is $K_1$-regular then examination of the Andersen spectral sequence reveals that $NK_2(A)\to\op{nil}K_2(A)\,(:=\ker(K_2(A)\to KV_2(A)))$ and $K_2(A)\to KV_2(A)$ are both surjective, whence there is a short exact sequence \[NK_2(A)\to K_2(A)\to KV_2(A)\to 0\]

\begin{corollary}\label{corollary_K_1_regularity_for_seminormal_rings}
Let $A$ be a one-dimensional, seminormal local ring, with residue field $k$. Then $A$ is $K_1$ regular if and only if each residue field of $\tilde A$ is a separable extension of $k$.
\end{corollary}
\begin{proof}
Set $B=\tilde A$ and $I=\frak m_A$. Then the data $I\subseteq A\subseteq B$ satisfy all the conditions above; moreover $K_1(A)=\mult A\to K_1(k)=\mult A$ is surjective, and so $K_2(A)\to K_2(k)$ is also surjective, by appealing to the fact that $K_2(k)$ is generated by Steinberg symbols.

Therefore we may apply the previous lemma, from which the result follows.
\end{proof}

A useful tool for describing $K_2$ of one-dimensional, seminormal local rings, and other `excision-like' examples, is the following theorem of R.~Dennis and M.~Krusemeyer \cite[Prop.~2.10 \& Thm.~3.1]{Dennis1979}:

\begin{theorem}\label{theorem_Dennis_Krusemeyer}
Let $A$ be a ring containing a subring $k$ and ideals $I_1,\dots,I_m$ such that $A\cong k\oplus I_1\oplus\dots\oplus I_m$ as abelian groups. Then \[K_2(A)\cong K_2(k)\oplus \bigoplus_{i=1}^n L_i\oplus \bigoplus_{i<j}(I_i/I_i^2\otimes_k I_j/I_j^2),\] where $L_i:=\ker(K_2(k+I_i)\to K_2(k))$.
\end{theorem}

We will need a variation of this theorem for Karoubi-Villameyor $K$-theory. Suppose that $A=k\oplus I_1\oplus\cdots\oplus I_m$ is as in the lemma, and suppose further that each ring $k+I_i$ is regular. We claim that \[KV_*(A)\cong KV_*(k)\oplus\bigoplus_{i=1}^m KV_*(k+I_i,I_i).\] By an obvious induction it is enough to treat the case $m=2$: Then
\begin{align*}
KV_*(A)
	&=KV_*(k+I_1)\oplus KV_*(A,I_2)\\
	&=KV_*(k+I_1)\oplus KV_*(k+I_2,I_2)\\
	&=KV_*(k)\oplus KV_*(k+I_1,I_1)\oplus KV_*(k+I_2,I_2),
\end{align*}
where the only non-trivial equality is the second, which follows from applying excision to $I_2\subseteq k+I_2\subseteq A$ (note that $A\to A/I_2$ is a $GL$-fibration since $k+I_1$ is regular). This completes the proof of the claim.
\end{appendix}

\bibliographystyle{acm}

\vspace{1cm}\noindent Matthew Morrow,\\
University of Chicago,\\
5734 S. University Ave.,\\
Chicago,\\
IL, 60637,\\
USA\\
{\tt mmorrow@math.uchicago.edu}\\
\url{http://math.uchicago.edu/~mmorrow/}\\
\end{document}